\RequirePackage{fix-cm}
\documentclass[11pt,reqno]{amsart}
\usepackage{amsmath,amssymb,amsthm}
\usepackage{amsfonts}
\usepackage{mathtools}
\usepackage{enumitem}

\numberwithin{equation}{section}

\theoremstyle{plain}
\newtheorem{theorem}{Theorem}[section]
\newtheorem{lemma}[theorem]{Lemma}

\newtheorem{proposition}[theorem]{Proposition}
\newtheorem{corollary}[theorem]{Corollary}

\newtheorem*{theorem-non}{Theorem}

\theoremstyle{definition}
\newtheorem{definition}[theorem]{Definition}

\newtheorem{remark}[theorem]{Remark}

\theoremstyle{definition}
\newtheorem{hypothesis}[theorem]{Hypothesis}

\theoremstyle{definition}

\newcommand{\be}{\begin{equation}}
\newcommand{\ee}{\end{equation}}
\newcommand{\n}{\noindent}

\DeclareMathOperator*{\esssup}{ess\,sup}
\DeclareMathOperator{\loc}{loc}
\DeclareMathOperator{\SO}{SO}
\DeclareMathOperator{\BMO}{BMO}
\DeclareMathOperator{\dist}{dist}

\DeclareMathOperator{\Div}{Div}
\DeclareMathOperator{\tr}{tr}
\DeclareMathOperator{\AM}{AM}
\DeclareMathOperator{\Var}{Var}
\DeclareMathOperator{\diag}{diag}

\newcommand{\SOn}{{\SO(n)}}

\DeclareMathOperator{\Lin}{Lin}
\DeclareMathOperator{\Psymn}{Psym_\mathit{n}}
\DeclareMathOperator{\Symn}{Sym_\mathit{n}}
\DeclareMathOperator{\Def}{Def}
\DeclareMathOperator{\AD}{AD}
\DeclareMathOperator{\trace}{trace}

\newcommand{\tsn}[1]{{[\kern-0.4ex] #1 [\kern-0.4ex]}}

\newcommand{\oo}{_{\mathrm o}}

\newcommand{\grad}{\nabla}

\newcommand{\bee}{{\mathbf e}}
\newcommand{\bff}{{\mathbf f}}

\newcommand{\Td}{{\;\!:\!\;}}

\newcommand{\ADu}{{\AD_u}}
\newcommand{\Varu}{{\Var_u}}

\newcommand{\bL}{{\mathbf L}}
\newcommand{\bB}{{\mathbf B}}

\newcommand{\bC}{{\mathbf C}}
\newcommand{\bi}{{\mathbf i}}

\newcommand{\bV}{{\mathbf V}}

\newcommand{\bU}{{\mathbf U}}
\newcommand{\bP}{{\mathbf P}}
\newcommand{\bE}{{\mathbf E}}
\newcommand{\bH}{{\mathbf H}}
\newcommand{\bN}{{\mathbf N}}
\newcommand{\bI}{{\mathbf I}}
\newcommand{\bF}{{\mathbf F}}
\newcommand{\bG}{{\mathbf G}}
\newcommand{\bK}{{\mathbf K}}
\newcommand{\bu}{{\mathbf u}}

\newcommand{\bv}{{\mathbf v}}
\newcommand{\bx}{{\mathbf x}}
\newcommand{\bd}{{\mathbf d}}
\newcommand{\by}{{\mathbf y}}
\newcommand{\bz}{{\mathbf z}}
\newcommand{\bn}{{\mathbf n}}
\newcommand{\bt}{{\mathbf t}}
\newcommand{\bw}{{\mathbf w}}
\newcommand{\bQ}{{\mathbf Q}}
\newcommand{\bS}{{\mathbf S}}
\newcommand{\bs}{{\mathbf s}}
\newcommand{\ba}{{\mathbf a}}
\newcommand{\bb}{{\mathbf b}}
\newcommand{\bA}{{\mathbf A}}
\newcommand{\bR}{{\mathbf R}}

\newcommand{\br}{{\mathbf r}}

\newcommand{\bm}{{\mathbf m}}

\newcommand{\sT}{{\mathcal T}}

\newcommand{\sB}{{\mathcal B}}
\newcommand{\sS}{{\mathcal S}}
\newcommand{\sD}{{\mathcal D}}

\newcommand{\sH}{{\mathcal H}}

\newcommand{\cO}{{\mathcal O}}

\newcommand{\mi}{{\text{-}}}

\newcommand{\R}{{\mathbb R}}
\newcommand{\A}{{\mathbb A}}

\newcommand{\Mn}{{\mathbb M}^{n\times n}}
\newcommand{\MN}{{\mathbb M}^{N\times n}}
\newcommand{\Mnp}{{\mathbb M}_+^{n\times n}}

\newcommand{\eee}{{_{\!\mathrm e}}}

\newcommand{\bFFe}{{{\mathbf F}\eee}}
\newcommand{\bFFei}{{{\mathbf F}_{\!\mathrm e}^{\mi1}}}
\newcommand{\bFFeiT}{{{\mathbf F}_{\!\mathrm e}^{\mi\rmT}}}
\newcommand{\bFFeT}{{{\mathbf F}_{\!\mathrm e}^{\rmT}}}
\newcommand{\bGG}{{\mathbf G}}
\newcommand{\bGGT}{{\mathbf G}^\rmT}

\newcommand{\bue}{{{\mathbf u}\e}}

\newcommand{\bzh}{{\widehat{\mathbf z}}}
\newcommand{\buh}{{\widehat{\mathbf u}}}
\newcommand{\bvh}{{\widehat{\mathbf v}}}
\newcommand{\bwh}{{\widehat{\mathbf w}}}

\newcommand{\Wu}{{W_{\!u\!}}}
\newcommand{\Au}{{{\mathbb A}_{u\!}}}
\newcommand{\Su}{{\bS_{u\!}}}
\newcommand{\bbu}{{\bb_{u\!}}}
\newcommand{\bsu}{{\bs_{u\!}}}

\newcommand{\Om}{{\Omega}}
\newcommand{\Ome}{{\Om_e}}
\newcommand{\Omc}{{\overline{\Omega}}}

\newcommand{\E}{{\mathcal E}}
\newcommand{\Eu}{{\E_u}}

\newcommand{\e}{_{\mathrm e}}

\newcommand{\h}{_{\mathrm h}}
\newcommand{\dd}{{\mathrm d}}
\newcommand{\DD}{{{\mathrm D}}}
\newcommand{\psihat}{\tilde{\psi}}
\newcommand{\phihat}{\tilde{\varphi}}

\newcommand{\rmT}{{\mathrm T}}

\def\Xint#1{\mathchoice
   {\XXint\displaystyle\textstyle{#1}}
   {\XXint\textstyle\scriptstyle{#1}}
   {\XXint\scriptstyle\scriptscriptstyle{#1}}
   {\XXint\scriptscriptstyle\scriptscriptstyle{#1}}
   \!\int}
\def\XXint#1#2#3{{\setbox0=\hbox{$#1{#2#3}{\int}$}
     \vcenter{\hbox{$#2#3$}}\kern-.5\wd0}}

\def\dashint{\Xint-}

\newlength{\extramargin}
\begin{document}

\title[Uniqueness of Equilibrium
with Small Strains in Elasticity]
{Uniqueness of Equilibrium with
Sufficiently Small Strains in Finite Elasticity}


\author[D. E. Spector]{Daniel E. Spector}
\address{National Chiao Tung University,
Department of Applied Mathematics,
Hsinchu, Taiwan}
\address{National Center for Theoretical Sciences,
National Taiwan University,
No. 1 Sec. 4 Roosevelt Rd.,
Taipei, 106, Taiwan}
\email{dspector@math.nctu.edu.tw}
\thanks{The first author (DS) is supported by the Taiwan Ministry of Science
and Technology under research grant 105-2115-M-009-004-MY2.}

\author[S. J. Spector]{Scott J. Spector}
\address{Department of Mathematics,
Southern Illinois University, Carbondale, IL 62901, USA}
\email{sspector@siu.edu}
\thanks{}


\date{11 November 2018}
\dedicatory{}

%
%


\begin{abstract}  The uniqueness of equilibrium for a compressible,
hyperelastic body subject to dead-load boundary conditions is
considered.  It is shown, for both the displacement and
mixed problems,
that there cannot be two solutions of the equilibrium
equations of Finite (Nonlinear) Elasticity whose nonlinear strains
are uniformly close to each other. This result is analogous to
the result of Fritz John (Comm.\ Pure Appl.\ Math.\ \textbf{25},
617--634, 1972) who proved that, for the displacement problem,
there is a most one equilibrium solution with uniformly small strains.
The proof in this manuscript utilizes Geometric Rigidity; a new
straightforward extension
of the Fefferman-Stein inequality to bounded domains;
and, an appropriate adaptation, for Elasticity,
of a result from the Calculus of Variations.
Specifically, it is herein shown that
the uniform positivity of the second variation of the energy
at an equilibrium solution implies that this mapping is
a local minimizer of the energy among deformations
whose gradient is sufficiently close, in $\BMO\cap\, L^1$,
to the gradient of the equilibrium solution.
\end{abstract}

\maketitle
\setlength{\parskip}{.5em}
\setlength{\parindent}{2em}
\baselineskip=15pt

\section{Introduction}\label{sec:Intro}


We herein consider the uniqueness of equilibrium
solutions for a compressible, hyperelastic body
$\overline{\Omega}\subset\R^n$, subject to dead loads.  This problem was
previously analyzed by John~\cite{Jo72} who showed that for the
pure-displacement (Dirichlet) problem there is at most one smooth solution
of the equilibrium  (Euler-Lagrange) equations among those mappings that
have uniformly small \emph{strains}:
\[  
\bE_\bu:= \tfrac12 \left[(\grad\bu)^\rmT\grad\bu-\bI\right],
\]  
\n where $\grad\bu$ denotes the matrix of partial derivatives of
$\bu:\overline{\Omega}\to\R^n$ and we write $\bF^\rmT$ for the transpose
of the $n$ by $n$ matrix $\bF$.
The main objective of this manuscript is the extension
of John's result to the mixed problem. However, our approach
also yields the uniqueness of equilibrium in a neighborhood in the
space of strains.   More precisely we prove that
given a smooth solution of the equilibrium equations,
$\bu\e$, at which the second variation of
the energy is uniformly positive, there is no other equilibrium
solution, $\bv\e$, for which the
difference of the two
right Cauchy-Green strain tensors:
\be\label{eqn:strain-diff-zero-1}
(\grad\bu\e)^\rmT\grad\bu\e- (\grad\bv\e)^\rmT\grad\bv\e
\ee
\n is uniformly small.


In the absence of body forces and surface tractions, the
total energy of a deformation $\bu:\overline{\Omega}\to\R^n$ of
a compressible, hyperelastic body is given by
\[
\E(\bu):=\int_\Omega W\big(\bx,\grad\bu(\bx)\big)\;\!\dd\bx,
\]
\n where $W:\overline{\Omega}\times\Mnp\to[0,\infty)$ denotes the
stored-energy density and we write $\Mnp$ for the set of $n$ by $n$
matrices with positive determinant.  We require that $\bu=\bd$ on $\sD$,
where $\bd$ is prescribed and  $\sD\subset\partial \Omega$ is nonempty
and relatively open.
The pure-displacement problem can then be expressed as the condition
$\sD=\partial \Omega$, while the genuine-mixed problem is
the condition $\sD\varsubsetneq\partial\Omega$.  We here
consider both problems.  With this notation,
we call $\bu\e$ an equilibrium solution
if it is a weak solution of the corresponding Euler-Lagrange equations:
\[
\delta\mathcal{E}(\bu\e)[\bw]
=
\int_\Omega
\bS\big(\bx,\grad\bu\e(\bx)\big)\!:\!\grad\bw(\bx)\, \dd\bx = 0
\]
\n for all $\bw\in W^{1,2}(\Omega;\R^n)$ that satisfy $\bw=\mathbf{0}$
on $\sD$, while the uniform positivity of the second variation of
$\E$ at $\bu\e$ is then the condition that
\[
\delta^2\E(\bu\e)[\bw,\bw]=\int_\Omega \grad\bw(\bx)
\!:\!\A\big(\bx,\grad\bu\e(\bx)\big)\big[\grad\bw(\bx)\big]\dd\bx
\ge
k\int_\Omega |\grad\bw(\bx)|^2\dd\bx
\]
\n for some $k>0$ and all $\bw\in W^{1,2}(\Omega;\R^n)$ that
satisfy $\bw=\mathbf{0}$ on $\sD$. Here we write
$W^{1,2}(\Omega;\R^n)$ for the usual
Sobolev space of square-integrable, vector-valued functions whose
distributional gradient, $\grad\bw$, is square integrable.
Also, $\bH\!:\!\bK:=\trace(\bH\bK^\rmT)$ and $\bS(\bx,\bF)$ and
$\A(\bx,\bF)$ denote the Piola-Kirchhoff stress
and the Elasticity Tensor, respectively:
\[
\bS(\bx,\bF):= \frac{\partial}{\partial\bF}W(\bx,\bF),
\qquad   \A(\bx,\bF):=
\frac{\partial^2}{\partial\bF^2}W(\bx,\bF).
\]





   It is well-known that when the second variation is uniformly positive
at an equilibrium solution $\bu\e$, then there is a neighborhood of
$\bu\e$ in the Sobolev space $W^{1,\infty}(\Omega;\R^n)$ in which there
are no other solutions of the equilibrium equations.  In addition, the
energy of any other mapping in this neighborhood is strictly greater than
the energy of $\bu\e$.  These assertions follow readily from a simple
analysis of the Taylor expansion of $\E$ that is inherited from the
Taylor series for the stored-energy function $W$:
\[
\E(\bw+\bu\e)= \E(\bu\e)
+ \delta\E(\bu\e)[\bw]
+\delta^2\E(\bu\e)[\bw,\bw]
+\mathcal{R}(\bu\e;\bw)
\]
\n with
\[
\big|\mathcal{R}(\bu\e;\bw)\big| \le
C \int_\Omega |\grad\bw(\bx)|^3\dd\bx.
\]


   In particular, the choice $\bw=\bv-\bu\e$, the fact that $\bu\e$ is
an equilibrium solution with uniformly positive second variation, and
the standard inequality
\be\label{eqn:Taylor-Linfinity}
\int_\Omega |\grad\bw(\bx)|^3\;\!\dd\bx
\le
||\grad\bw||_{L^\infty(\Om)}\int_\Omega |\grad\bw(\bx)|^2\;\!\dd\bx
\ee
\n imply that
\be\label{eqn:unique-minimizer}
\E(\bv)\ge \E(\bu\e)+ c\int_\Omega |\grad\bw(\bx)|^2\;\!\dd\bx
\ee
for some $c>0$, provided
$||\grad\bw||_{L^\infty(\Om)}$ is sufficiently
small.  From this one deduces the claims.


The essential point of John's work is that, while the assumption that
$||\bE_\bu||_{L^\infty(\Om)}$ and $||\bE_\bv||_{L^\infty(\Om)}$
are small need not imply the same for
$||\grad\bu-\grad\bv||_{L^\infty(\Om)}$, the above argument can be
suitably modified to obtain uniqueness for the pure-displacement problem.
For the purpose of our work it is convenient for us to separate two
key components of his proof.  The first is the fact that
uniformly small strains
$\bE_\bu$ and $\bE_\bv$ imply that $\grad\bu-\grad\bv$ has small norm
in the space of functions of Bounded Mean Oscillation, a
Geometric-Rigidity result that was obtained by John in various
forms~\cite{Jo61,Jo72,Jo72-2} and which has been further studied by
Friesecke, James, \& M{\"u}ller~\cite{FJM02} (see, also,
Kohn~\cite{Ko82} and Conti \& Schweizer~\cite{CS06}).
The second is that, while the preceding argument culminating in
inequality \eqref{eqn:unique-minimizer} is designed for $L^\infty$
neighborhoods of the gradient, it extends to $\BMO$ neighborhoods,
although this requires a more sophisticated analysis. Specifically,
one requires tools that allow for the replacement of $L^\infty$ by
$\BMO$. The canonical example of such a tool is the John-Nirenberg
inequality \cite{JN61}, and indeed, this is precisely what John used
in his proof of uniqueness.




In this paper we pursue an alternative approach to this replacement
through a local analogue of an inequality of Fefferman \&
Stein~\cite{FS72} for bounded Lipschitz domains.
In particular, we make use of results of Iwaniec~\cite{Iw82} and Diening,
R$\overset{_\circ}{\mathrm{u}}$\v zi\v cka, \& Schumacher~\cite{DRS10}
to obtain, in Theorem~2.6, an inequality that is valid for any bounded
Lipschitz domain $\Omega$: For every $q\in(1,\infty)$ there is a
constant $F=F(q)>0$ such that any $\psi\in L^1(\Omega)$ that satisfies
$\psi_\Omega^{\#}\in L^q(\Omega)$
will also satisfy
\[  
F^{\mi1}\int_\Omega |\psi|^q\,\dd\bx
\le
\int_\Omega |\psi_\Omega^{\#}|^q\,\dd\bx
+ \Big|\,\int_\Omega \psi \,\dd\bx \Big|^q.
\]  
\n Here $\psi_\Omega^{\#}$ (see \eqref{eqn:global-max}$_2$)
denotes the maximal function of Fefferman \& Stein~\cite{FS72}.
This inequality implies an interpolation inequality analogous
to \eqref{eqn:Taylor-Linfinity} (as well as a more general family
of inequalities, see Section~\ref{sec:MF}):
\be\label{eqn:small-BMO-zero-2}
\| \grad\bw\|_{L^3(\Omega)}
\le J
\Big(\tsn{\grad\bw}_{\BMO(\Om)}
+\Big|\int_\Om\grad\bw\,\dd\bx\Big|\,\Big)^{\!1/3}
\| \grad\bw\|^{2/3}_{L^2(\Omega)},
\ee
\n where $\tsn{\grad\bw}_{\BMO(\Om)}$ denotes the seminorm of $\grad\bw$
in $\BMO(\Omega)$ 
(see \eqref{eqn:BMO-V-norm}).
%
%
Therefore, if we replace \eqref{eqn:Taylor-Linfinity} by
\eqref{eqn:small-BMO-zero-2}, we obtain, in Theorem~\ref{thm:SVP=LMBMO},
a general uniqueness theorem
in the Calculus of Variations for neighborhoods where both
$\tsn{\grad\bw}_{\BMO(\Om)}$ and $\int_\Om\grad\bw\,\dd\bx$ are
small.   This result is in the spirit of a theorem of Kristensen
\& Taheri~\cite{KT03} (see, also, Campos Cordero~\cite{Ca17} and
Firoozye~\cite{Fi92})
for the Dirichlet problem under the assumption that the extension of
$\grad \bw$ by zero is small as an element of $\BMO(\R^n)$.

With these results established, we can return to the question of
uniqueness in Elasticity.  In particular, let us observe that for
the pure-displacement problem, an integration by parts shows that the
integral in \eqref{eqn:small-BMO-zero-2} is zero; thus, the coefficient
(with exponent $1/3$)
in the right-hand side of \eqref{eqn:small-BMO-zero-2} reduces to the
$\BMO(\Omega)$-seminorm, whose smallness follows from
Geometric Rigidity, and so we obtain John's result.  For the mixed
problem, with a few elementary calculations we show that for functions
which agree on a portion of the boundary one actually has a closeness
not just of the seminorms, but of the entire norm of their derivatives
in the space $\BMO(\Omega)\cap\,L^1(\Om)$.  Thus we obtain uniqueness
for the mixed
problem for small-strain solutions.  The general result asserted at
the beginning of the introduction then follows by a change of variables to
the deformed configuration and an application of the previous analysis.


As noted by Kohn~\cite[p.~134]{Ko82}, there is the question of whether
one has an existence theory that produces an equilibrium solution with
uniformly small
strains.  In particular, it is not clear from the existence theory of
Ball~\cite{Ba77}, or any of its many extensions, whether or not
$\bE_\bu$ is uniformly small.  A few things can be said in this regard.
First, a result of
Zhang~\cite{Zh91} for the displacement problem shows that
Ball's minimizer is the equilibrium solution obtained from
the implicit function
theorem\footnote{Although the results in \cite{Va88} are only stated
for the pure-displacement and pure-traction problems, it appears that
a similar analysis will be valid for the mixed problem \emph{provided}
that the parts of the boundary where displacements and tractions are
prescribed have disjoint closures, for example, the inside and the
outside of a thick spherical shell.
See, e.g., Ciarlet~\cite[Chapter~6]{Ci1988}.} (see, e.g.,
Valent~\cite{Va88} or Ciarlet~\cite[Chapter~6]{Ci1988})
provided the boundary is smooth
and the boundary displacements are sufficiently small.
Second, the equilibrium solution obtained from
the implicit function theorem will be as smooth
as desired when the boundary, the stored energy $W$,
and the boundary displacement $\bd$
are all sufficiently smooth.
Finally, unconditional uniqueness of equilibrium solutions
is neither desired nor expected in Nonlinear Elasticity.
For example, when a thin rod is subjected to uniaxial compression,
one expects that the rod will buckle and that there will be more than
one buckled equilibrium solution. Thus it may be natural to impose
additional restrictions to obtain uniqueness.


Results in the literature have established local uniqueness,
uniqueness when the deformation gradient lies in certain subsets of
$\Mnp$, uniqueness of the absolute minimizer of the energy when
appropriate extra conditions are imposed, and
uniqueness of equilibrium solutions for the displacement problem
for certain bodies and boundary values.
In particular,
Knops and Stuart~\cite{KS84} (see, also, Bevan~\cite{Be11} and
Taheri~\cite{Ta03}) have proven that, for a star-shaped body,
the homogeneous deformation $\bu\h(\bx)=\bF\bx+\ba$
is the only smooth equilibrium solution that satisfies a homogeneous
displacement boundary condition whenever the energy is
globally rank-one convex and strictly quasiconvex at $\bu\h$.
Gurtin and Spector~\cite{GS79} have shown that there is at most one
solution of the equilibrium equations that lies in any convex set where
the second variation of the energy is strictly positive.  Gao, Neff,
Roventa, and Thiel~\cite{GNRT17} have recently established that
the convexity of the elastic energy, when considered as a function
of the right Cauchy-Green strain tensor, implies that any
equilibrium solution $\bu\e$, at which the Cauchy Stress
is positive semi-definite at every point, is an
absolute minimizer of the energy.  Moreover, if in addition
$\bC\e(\bx)$
is a point of strict convexity of the energy at every
$\bx\in\Omega$, then $\bu\e$ is the
unique absolute minimizer of the energy.
Sivaloganathan \& Spector~\cite{SS18} have demonstrated that, for a
large class of polyconvex stored-energy functions, an equilibrium
solution that satisfies a certain pointwise inequality is the unique
absolute minimizer of the energy.  They also gave an elementary proof,
for the pure-displacement problem, of John's uniqueness with small
strains result that we consider in Section~\ref{sec:Def-SS}.


There is also an extensive
literature on nonuniqueness in Nonlinear Elasticity.
Post \& Sivaloganathan~\cite{PS97} have proven that there are an
infinite number of equilibrium solutions for certain displacement
problems for an annulus.    Antman~\cite{An79} has
shown that, for the pure-traction problem, a thick spherical shell
without loads has a second
equilibrium solution corresponding to an everted deformation.
See \cite[footnote~3]{SS18} for additional references that contain
examples of nonuniqueness.




Let us mention some related open problems before we proceed to the
plan of the paper.   Although our technique could, in principle, be
applied to the pure-traction problem, we have not considered dead-load
tractions applied to the entire boundary since the lack of any displacement
boundary condition necessitates an additional mathematical constraint
that induces the gradients of two solutions to be close in $L^1$
(see Proposition~\ref{prop:BC+SR=close-in-L1}). From a physical
point of view the difficulty is a potential axis of
equilibrium for the loads that leads to nonuniqueness of
equilibrium solutions.  For a
detailed explanation see, e.g., Valent~\cite[Chapter~5]{Va88} or
Truesdell \& Noll~\cite[\S44]{TN65} and the references therein.
An extension of our results to incompressible elastic bodies is
of interest.  Difficulties include the constraint that the
deformation gradient lie on the manifold $\det\grad\bu=1$ and the
pressure, which  appears as a Lagrange multiplier in the
equilibrium equations.
A uniqueness result for live loads would also be of interest.
Here one might
want to look at \cite{CPV03,PV90,Se67,Sp80},  \cite[\S2.7]{Ci1988},
or \cite[\S13.3]{Si97}.
Lastly, our results necessitate that the equilibrium equations
have a solution $\bu\e$ that is Lipschitz continuous.\footnote{The
standard existence theory for Nonlinear Elasticity
(see, e.g., Ball~\cite{Ba77}) yields
minimizers in $W^{1,p}(\Omega;\R^n)$ that satisfy only alternative forms of
the equilibrium equations.  See, e.g., Ball~\cite[Theorem~2.4]{Ba02}.}
However, some of our results also require that $\bu\e$ be one-to-one
on $\overline{\Omega}$, which prohibits self-contact of the boundary
of $\bu\e(\Omega)$.  It would be of interest if this assumption could be
excluded.  
In this regard, see Remark~\ref{rem:final-last}.


We commence our investigation in Section~\ref{sec:MF} with a development
of the requisite Harmonic Analysis.  In particular,
after we recall some properties of the Hardy-Littlewood and
Fefferman-Stein maximal functions, we establish a local analogue of
Fefferman and Stein's inequality  
in Theorem~\ref{thm:main-1}.  We then demonstrate, in
Theorem~\ref{thm:main-2}, how this inequality
gives rise to a family of interpolation inequalities that implies
\eqref{eqn:small-BMO-zero-2}.


In Section~\ref{sec:NLP-CV} we first recall some background material
from the Calculus of Variations.   We then make use of
the interpolation inequality from the previous section to establish
two results.  The first, Lemma~\ref{lem:extra},  shows that
whenever two mappings, $\bu$ and $\bv$, have gradients that are
sufficiently close in $\BMO\cap\,L^1$, the uniform positivity of the
second variation of the energy at either mapping implies that
the second variation at the other mapping is strictly positive in the
direction $\bw=\bv-\bu$, a simple result that we have found to be
helpful in establishing uniqueness of equilibrium solutions.
Finally, we show, in Theorem~\ref{thm:SVP=LMBMO},
that any mapping whose gradient is
sufficiently close, in $\BMO\cap\,L^1$, to the gradient of a Lipschitz
solution of the Euler-Lagrange equations
whose second variation is
uniformly positive, will have strictly greater energy
than the solution and also cannot satisfy
the Euler-Lagrange equations.


In Section~\ref{sec:RGRSM}
we observe that a general version of the
relationship between the distance from $\grad\bu$ to the set of
rotations and the norm in $\BMO(\Omega)$ of $\grad\bu$ is a consequence
of a Geometric-Rigidity result established in \cite{CS06,FJM02}:
Given a mapping
$\bu\in W^{1,p}(\Omega;\R^n)$, $1<p<\infty$, there is a particular
rotation $\bR_\bu$ such that the distance in $L^p(\Omega)$
from $\grad\bu$ to $\bR_\bu$ is, up to a
multiplicative constant which does not depend on $\bu$, a lower
bound for the distance in $L^p(\Omega)$ from $\grad\bu$ to the
set of rotations (see Proposition~\ref{prop:GR}).
It follows that, when $\grad\bu$ is uniformly close to the set of
rotations, $\grad\bu$ is small in $\BMO(\Omega)$.
We further show in Proposition~\ref{prop:BC+SR=close-in-L1} that
two mappings in $W^{1,p}(\Omega;\R^n)$, $p>n$, that share the
same boundary values on $\sD$
will have gradients that are close in $L^1(\Omega)$
whenever the gradients are close to the set of rotations.


In Section~\ref{sec:NLE}  we first recall some of the terminology from
Continuum Mechanics: bodies, deformations, deformation gradients,
strains, and the elastic energy and its first two derivatives: the
Piola-Kirchhoff stress tensor and the Elasticity Tensor.  We then show, in
Lemma~\ref{lem:strain=dist-to-rotations}, that the strains $\bE_\bu$
are uniformly small if and only if the gradient of the underlying
deformation $\bu$ is uniformly close to the
set of rotations. Finally, we note, in Theorem~\ref{thm:NT-3.0},
that when the coefficient (with exponent $1/3$)
in the right-hand side of \eqref{eqn:small-BMO-zero-2}
is small and $\bu\e$ is
an equilibrium solution with uniformly positive second variation, then
$\bv:=\bu\e+\bw$ cannot be a solution of the equilibrium equations and
$\bv$ must also have strictly greater energy than $\bu\e$.


In Section~\ref{sec:Def-SS}
we present our uniqueness results for Nonlinear Elasticity
when all strains are
uniformly small.  We first establish that when the reference
configuration
is stress free and the Elasticity Tensor at the reference configuration
is strictly positive definite, then not only is the second variation
uniformly positive at the reference
configuration, a result that is well-known and which follows from
Korn's inequality, but the second variation
is uniformly positive at any smooth deformation with
sufficiently small strains $\bE_\bu$.   We then obtain, in
Theorem~\ref{thm:unique},
the result mentioned in the first paragraph of this manuscript:
There is at most one equilibrium solution $\bu\e$ with sufficiently
small strains $\bE\e$ and, moreover, any other deformation with
small strains has
strictly greater energy than the energy of $\bu\e$.


Finally, in Section~\ref{sec:CRC},
we extend our results for Elasticity to include one mapping
with potentially large strains and a second mapping that is
close to it in the space of strains, that is, for which the
quantity given in \eqref{eqn:strain-diff-zero-1} is uniformly small.
We prove that given an equilibrium solution $\bu\e$ that is a
diffeomorphism and for which the second variation of the energy is
uniformly positive, any other mapping, $\bv$, with
right Cauchy-Green tensor, $\bC_\bv=(\grad\bv)^\rmT\grad\bv$,
uniformly and sufficiently close to $\bC\e=(\grad\bu\e)^\rmT\grad\bu\e$
cannot be a solution of the equilibrium equations and $\bv$ must also
have strictly
greater energy than the energy of $\bu\e$.  Our proof involves a change
of variables that replaces the reference configuration $\Omega$ by
the deformed configuration $\bu\e(\Omega)$.  Once this is accomplished,
small modifications of our previous analysis then yield
the desired result.


\part*{Part I: Maximal Functions, the Second Variation,
and $\BMO$ Local Minimizers}\label{part:max+CV}

\section{Maximal Functions}\label{sec:MF}


In this section we first recall some of the properties of the
Hardy-Littlewood and Fefferman-Stein maximal functions.  We then show
that results of Iwaniec~\cite{Iw82} and Diening,
R$\overset{_\circ}{\mathrm{u}}$\v zi\v cka, \&  Schumacher~\cite{DRS10}
yield a version of the Fefferman-Stein
inequality that is valid for many bounded, open regions.  This
inequality then allows us to give an elementary proof of a result of
John~\cite[p.~632]{Jo72} that bounds the $L^q$-norm of a function by its
$L^p$-norm, $q>p$, when the function is sufficiently small in
$\BMO\cap\,L^1$ rather than $L^\infty$.

\subsection{Preliminaries}\label{sec:prelim-maximal}

For any domain (nonempty, connected, open set)
$\Omega\subset\R^n$, $n\ge2$, we denote by
$L^p(\Omega)$, $p\in[1,\infty)$, the space
of real-valued Lebesgue measurable functions, $\psi$,
whose $L^p$-norm is finite:
\[
||\psi ||^p_{p,\Omega} := \int_\Omega |\psi(\bx)|^p\,\dd\bx < \infty.
\]
\n $L^1_{\loc}(\Omega)$ will consist of those Lebesgue measurable
functions that are integrable on every compact subset of $\Omega$.
$L^\infty(\Omega)$ will denote those Lebesgue measurable
functions whose essential supremum is finite.
Given any $\psi\in L^1_{\loc}(V)$, where $V=\R^n$ or $V$ is
a bounded domain, the \emph{Hardy-Littlewood} and
\emph{Fefferman-Stein maximal functions} of $\psi$ are given by
\be\label{eqn:global-max}
\psi_V^{\star} (\bx)
:=
\sup_{\substack{Q\ni \bx,\\ Q\subset V}}\
\dashint_Q |\psi(\by)|\,\dd\by,
\qquad
\psi_V^{\#} (\bx)
:=
\sup_{\substack{Q\ni \bx,\\ Q\subset V}}\
\dashint_Q \big|\psi(\by)-\langle\psi\rangle_Q\big|\,\dd\by,
\ee
\n respectively.  When $V=\R^n$ we shall omit the subscript $V$.
Here, and in the sequel, \emph{the symbol} $Q$ \emph{will
denote a nonempty, bounded (open) $n$-dimensional
hypercube}\footnote{We shall henceforth refer to a
$Q$ as a \emph{cube}, rather than a hypercube or square.} \emph{with
faces parallel to the coordinate hyperplanes} and
\[
\langle\psi\rangle_V
:=\dashint_V \psi(\bx)\,\dd\bx
:= \frac1{|V|}\int_V \psi(\bx)\,\dd\bx,
\]
\n \emph{the average value of} $\psi$,  where
$|V|$ denotes the $n$-dimensional Lebesgue measure of any
bounded domain $V\subset\R^n$.  For future reference we note that it
is not difficult to show that these functions satisfy
the pointwise estimates,\footnote{The first estimate follows from the
Lebesgue point theorem, while the second estimate is a straightforward
consequence of \eqref{eqn:global-max}.
See, e.g., \cite[p.~95]{Gr08} and \cite[p.~184]{Gr09}.}
for $a.e.~\bx\in V$,
\be\label{eqn:pointwise-bound}
\psi(\bx) \le \psi_V^{\star} (\bx),
\qquad
\psi_V^{\#} (\bx)\le 2\psi_V^{\star} (\bx).
\ee


   The $\BMO$-seminorm is given by
\be\label{eqn:BMO-V-norm}
\tsn{\psi}_{\BMO(V)}:=  \sup_{Q\subset V}
 \dashint_Q |\psi(\bx)-\langle\psi\rangle_Q|\,\dd\bx,
\ee
\n while the space $\BMO(V)$ (Bounded Mean Oscillation) is defined by
\[
\BMO(V)
:=
\{\psi\in L_{\loc}^1(V): \tsn{\psi}_{\BMO(V)} < \infty\}.
\]
\n  Here, once again, $V=\R^n$ or $V$ is a bounded domain and we shall
omit the $V$ when $V=\R^n$. For future reference we note that
\be\label{eqn:sharp=BMO}
\tsn{\psi}_{\BMO(V)}
=
\big\|\psi_V^{\#}\big\|_{\infty,V}
=
\esssup_{\bx\in V}\psi_V^{\#}(\bx).
\ee


   Suppose now that $V$ is a bounded domain and $\varphi\in L^p(V)$,
$p\in[1,\infty]$. Then we define its \emph{extension}
$\phihat:\R^n\to\R$, to all of $\R^n$, by
\[  
\phihat(\bx):=
\begin{cases}
\varphi(\bx), & \text{if $\bx\in V$},\\
0, & \text{if $\bx\notin V$}.
\end{cases}
\]  
\n Clearly,  $\phihat\in L^p(\R^n)$.   Moreover,
the Hardy-Littlewood maximal function, $\phihat^{\star}$, is given by
\be\label{eqn:HL-max-2}
\phihat^{\star} (\bx)=(\phihat)^{\star} (\bx)
:=
\sup_{\substack{Q\ni \bx,\\ Q\subset \R^n}}\
\frac1{|Q|}\int_{Q\cap V} |\varphi(\by)|\,\dd\by,\quad  \bx\in \R^n.
\ee
\n In the sequel, we shall make use of a result of
Hardy \& Littlewood and Wiener.
\begin{proposition}\label{prop:HL-inequality}
\emph{(See, e.g., \cite[p.~88]{Gr08}, \cite[p.~5]{Stein-1970}
or \cite[p.~13]{Stein-1993}.)}
Let $1< p \le \infty$.  Then there exists a constant\footnote{Although
most of the constants in this manuscript will depend on both
the dimension $n$ and the domain, we shall usually omit this
dependence in order
to simplify the exposition. However, $H$
does not depend on $n$.}  $H=H(p)$
such that if $\psi\in L^p(\R^n)$, then
$\psi^{\star}\in L^p(\R^n)$ and $\psi$ and $\psi^{\star}$ satisfy
\be\label{eqn:HL-Rn-1}
||\psi^{\star}||_{p,\R^n} \le  H ||\psi||_{p,\R^n}.
\ee
\end{proposition}


We shall also utilize a more recent result of
Diening, R$\overset{_\circ}{\mathrm{u}}$\v zi\v cka, \&  Schumacher.
\begin{proposition}\label{prop:DRS-inequality}
\emph{(\cite[Theorem~5.23]{DRS10})}
Let $q\in(1,\infty)$ and suppose that $U\subset\R^n$ is
a Lipschitz or John
domain.\footnote{See Remark~\ref{rem:Lipschitz}.}
Then there exists a constant $R=R(q,U)$ with the following property:
If $\psi\in L^1(U)$  and  $\psi_U^{\#}\in L^q(U)$, then
$\psi\in L^q(U)$ and
\be\label{eqn:DRS}
\dashint_U \big|\psi- \langle\psi\rangle_U\big|^q \,\dd\bx
\le
R\dashint_U \big|\psi_U^{\#}\big|^q\,\dd\bx.
\ee
\end{proposition}


\begin{remark}\label{rem:Iwaniec} (1).~A key ingredient in the proof of
Proposition~\ref{prop:DRS-inequality} is a result of
Iwaniec~\cite[Lemma~4]{Iw82} that establishes a version of the
Fefferman-Stein~\cite[Theorem~5]{FS72} inequality when the
domain is a cube. (2).~As noticed in \cite{DRS10}, if
$\langle\psi\rangle_U=0$, then \eqref{eqn:DRS} together with
\eqref{eqn:HL-Rn-2} shows that the original Fefferman-Stein inequality
is also valid for certain bounded domains.
\end{remark}

\begin{remark}\label{rem:Lipschitz} (1).~By a
\emph{Lipschitz domain} $U$
we mean a bounded domain whose boundary $\partial U$ is
(strongly) Lipschitz.
See, e.g., \cite[p.~127]{EG92}, \cite[p.~72]{Mo66}, or
\cite[Definition~2.5]{HMT07}.   Essentially, a bounded domain is
Lipschitz if, in a neighborhood of every boundary point,
the boundary is the graph of a
Lipschitz function and the domain is on ``one side'' of this graph.
(2).~Proposition~\ref{prop:DRS-inequality} is
valid for a class of domains that is larger than
Lipschitz domains:~\emph{John domains}~\cite{Jo61}.
Roughly speaking, in a John domain there is a particular point
that can be connected
to every other point by a rectifiable curve; these curves
have uniformly bounded length; and the curves do not get ``too
close to the boundary.''
See, e.g., \cite{DRS10} or \cite{MS79} for a precise description.
\end{remark}


\subsection{Some Properties of Maximal Functions on Bounded Domains}

\subsubsection{Extensions of Results Previously
Established on Cubes (and on $\R^n$)}

A well-known result is that the Hardy-Littlewood-Wiener inequality,
\eqref{eqn:HL-Rn-1}, is also valid on every bounded domain.  We present
a proof for the convenience of the reader.

\begin{lemma}\label{lem:HL-B-domains}
Let $V\subset\R^n$ be a bounded domain and suppose that
$p\in(1,\infty]$.  Then there exists a constant
$H=H(p)>0$ such that if $\psi\in L^p(V)$, then
$\psi_V^{\star}\in L^p(V)$ and $\psi$ and $\psi_V^{\star}$ satisfy
\be\label{eqn:HL-Rn-2}
||\psi_V^{\star}||_{p,V} \le  H ||\psi||_{p,V}.
\ee
\end{lemma}

\begin{proof}[Proof for $p\ne\infty$]  Fix $p\in(1,\infty)$
and let $\psi\in L^p(V)$. Then,
since $\psihat=0$ on $\R^n\setminus V$,
$\psihat\in L^p(\R^n)$ with
\be\label{eqn:one}
\int_{\R^n} |\psihat|^p\,\dd\bx = \int_V |\psi|^p\,\dd\bx.
\ee
\n Therefore, we can apply Proposition~\ref{prop:HL-inequality}
to conclude,
with the aid of \eqref{eqn:one}, that $\psihat^{\star}\in L^p(\R^n)$ and
\be\label{eqn:two}
\int_{\R^n} |\psihat^{\star}|^p\,\dd\bx \le H^p\!\int_V |\psi|^p\,\dd\bx.
\ee


  The definitions of $\psi_V^{\star}$
and $\psihat^{\star}$, \eqref{eqn:global-max}$_1$
and \eqref{eqn:HL-max-2}, imply that
\[
\psi_V^{\star} (\bx) \le \psihat^{\star} (\bx)\
\text{ for $a.e.~\bx\in V$}
\]
\n and hence
\be\label{eqn:final}
\int_V |\psi_V^{\star}|^p\,\dd\bx
\le
\int_{\R^n} |\psihat^{\star}|^p\,\dd\bx.
\ee
\n The desired result, \eqref{eqn:HL-Rn-2},  now follows from
\eqref{eqn:two} and \eqref{eqn:final}.
\end{proof}

   We next establish a local version of the
Fefferman-Stein inequality that is valid on
certain bounded domains.


\begin{theorem}\label{thm:main-1}
Let $q\in(1,\infty)$ and suppose that
$U\subset \R^n$ is a Lipschitz (or John) domain.
Then there exists a constant $F=F(q,U)>0$ such that every
$\psi\in L^1(U)$ that satisfies $\psi_U^{\#}\in L^q(U)$
will also satisfy $\psi\in L^q(U)$ with
\be\label{eqn:Fefferman-Stein}
\dashint_U |\psi|^q\,\dd\bx
\le
F\!\left(\, \dashint_U \big|\psi_U^{\#}\big|^q\,\dd\bx
+ \Big|\,\dashint_U \psi \,\dd\bx \Big|^q\right).
\ee
\end{theorem}

Before we prove Theorem~\ref{thm:main-1}, we first note that
if we combine it with the Hardy-Littlewood-Wiener
inequality on bounded domains, \eqref{eqn:HL-Rn-2},
we find that a result similar to
Iwaniec's~\cite{Iw82} version of the
Fefferman-Stein inequality for cubes is also valid for
John domains (except for the case $q=1$).

\begin{corollary}\label{cor:FS-I-1}\label{thm:FS-B-domains}
Let $q\in(1,\infty)$ and suppose that
$U\subset \R^n$ is a Lipschitz (or John) domain.
Then there exists a constant $S=S(q,U)>0$ such that every
$\psi\in L^1(U)$ that satisfies $\psi_U^{\#}\in L^q(U)$
will also satisfy $\psi_U^{\star} \in L^q(U)$ with
\[
\dashint_U |\psi_U^{\star}|^q\,\dd\bx
\le
S\!\left(\, \dashint_U \big|\psi_U^{\#}\big|^q\,\dd\bx
+ \Big|\,\dashint_U \psi\,\dd\bx \Big|^q\right).
\]
\end{corollary}


\begin{proof}[Proof~of~Theorem~\ref{thm:main-1}]  Fix $q\in(1,\infty)$.
Then, by the triangle inequality,
\be\label{eqn:psi-1}
|\psi(\bx)|
\le
|\psi(\bx)- \langle\psi\rangle_U|+|\langle\psi\rangle_U|.
\ee
\n Thus, if we take \eqref{eqn:psi-1} to the $q$-th power,
use the
standard inequality $|a+b|^q\le 2^{q-1}(|a|+ |b|)$,
integrate over $U$, and divide by $|U|$ we find that
\be\label{eqn:psi-2}
\dashint_U |\psi|^q\,\dd\bx
\le
2^{q-1}\dashint_U |\psi- \langle\psi\rangle_U|^q \,\dd\bx
+
2^{q-1}\Big| \dashint_U \psi\,\dd\bx\Big|^q.
\ee
\n Finally, Proposition~\ref{prop:DRS-inequality}
yields a constant $R=R(q)>0$, which does
not depend on $\psi$, such that
\be\label{DRS-1}
\dashint_U |\psi- \langle\psi\rangle_U|^q \,\dd\bx
\le
R\, \dashint_U \big|\psi_U^{\#}\big|^q\,\dd\bx.
\ee
The desired result, \eqref{eqn:Fefferman-Stein},
now follows from \eqref{eqn:psi-2} and   \eqref{DRS-1}.
\end{proof}


\subsubsection{An Application of the Local Fefferman-Stein
Inequality}\label{sec:AFS}

We next utilize Theorem~\ref{thm:main-1} to establish an
interpolation inequality that will be important when we consider
local minimizers of an integral functional in Section~\ref{sec:NLP-CV}.

\begin{theorem}\label{thm:main-2}  Let $U\subset\R^n$ be a
Lipschitz (or John) domain.  Then, for all $q\in[1,\infty)$,
\[
\BMO(U)\cap\,L^1(U)\subset L^q(U)
\]
\n with continuous injection, i.e., there is a constant
$J_1=J_1(q,U)>0$ such that,
for every $\psi\in\BMO(U)\cap\,L^1(U)$,
\be\label{eqn:BN}
\bigg(\,\dashint_U |\psi|^q\,\dd\bx\bigg)^{\!\!1/q}\!
\le
J_1\|\psi\|_{\BMO(U)\cap\,L^1(U)}.
\ee               
\n Moreover, if  $1\le p<q<\infty$
then there exists a
constant $J_2=J_2(p,q,U)>0$ such that every
$\psi\in\BMO(U)\cap\,L^1(U)$ satisfies
\be\label{eqn:RH}
||\psi||_{q,U}
\le
J_2\Big(||\psi||_{\BMO(U)\cap\,L^1(U)}\Big)^{1-p/q}
\Big(||\psi||_{p,U}\Big)^{p/q}.
\ee
\n In addition, the constants $J_i$ are scale invariant, that is,
$J_i(\lambda U +\ba)=J_i(U)$ for every $\lambda>0$ and $\ba\in\R^n$.
Here (see (\ref{eqn:BMO-V-norm}))
\be\label{eqn:BMO-norm}
\|\psi\|_{\BMO(U)\cap\,L^1(U)}:=\tsn{\psi}_{\BMO(U)}
+\Big|\dashint_U \psi\,\dd\bx\Big|.
\ee
\end{theorem}


\begin{remark} (1).~Inequality \eqref{eqn:BN} with $q=1$ shows that
\eqref{eqn:BMO-norm} is an equivalent norm on  $\BMO(U)\cap\,L^1(U)$;
this inequality was previously established by  Brezis \&
Nirenberg~\cite[Lemma~A.1]{BN95} for connected,
compact Riemannian manifolds without boundary.
(2).~Inequality \eqref{eqn:RH}, for a function that has integral
equal to zero and is sufficiently small in $\BMO(Q)$ ($Q$ a cube),
was obtained by John~\cite[p.~632]{Jo72}, who showed that it is a
consequence of the John-Nirenberg inequality~\cite{JN61}.
\end{remark}

\begin{remark}  Our proof of \eqref{eqn:RH} 
makes use of Theorem~\ref{thm:main-1}. However, since it is
an interpolation inequality, there are other techniques one might use.
In particular, there is an analogue of \eqref{eqn:RH} for $\R^n$
(see, e.g., Bennett \& Sharpley~\cite[Theorem~8.11]{BS88}) and so
one might try to combine P.~Jones' extension
theorem~\cite{Jo82} with such an inequality.
One could also consider an approach that employs complex interpolation
theory on metric measure spaces.  In this regard see
Carbonaro, Mauceri \& Meda~\cite{CMM09,CMM10}.
\end{remark}


\begin{proof}[Proof of Theorem~\ref{thm:main-2}] Fix $q\in (1,\infty)$
and suppose that $\psi\in\BMO(U)\cap\,L^1(U)$. Then
\eqref{eqn:sharp=BMO} gives us $\psi_U^{\#}\in L^\infty(U)$ and
consequently
Theorem~\ref{thm:main-1} yields $\psi\in L^q(U)$ and
a constant $F=F(q,U)>0$, which does not depend on $\psi$, such that
\be\label{eqn:Fefferman-Stein-two}
F^{\mi1}\dashint_U |\psi|^q\,\dd\bx
\le
\dashint_U \big|\psi_U^{\#}\big|^q\,\dd\bx
+ \Big|\,\dashint_U \psi \,\dd\bx \Big|^q.
\ee
\n If we now make use of \eqref{eqn:sharp=BMO} we find that
\eqref{eqn:Fefferman-Stein-two} implies that
\be\label{eqn:Fefferman-Stein-three}
F^{\mi1}\dashint_U |\psi|^q\,\dd\bx
\le
\tsn{\psi}_{\BMO(U)}^{\,q}
+ \Big|\,\dashint_U \psi \,\dd\bx \Big|^q.
\ee
\n Inequality \eqref{eqn:BN} for $q>1$ and with
$J_1:=\sqrt[q]{2F(q,U)\,}$  now follows  from
 \eqref{eqn:BMO-norm},
\eqref{eqn:Fefferman-Stein-three},
and the standard inequality $|a|^q+|b|^q\le 2 (|a|+|b|)^q$.
Inequality \eqref{eqn:BN} with $q=1$ is a consequence of
H\"olders inequality and \eqref{eqn:BN} with $q>1$.
The scale invariance of $J_1$ is clear from \eqref{eqn:BN}
and \eqref{eqn:BMO-norm}, since the average value of any
function is scale invariant.

We next establish \eqref{eqn:RH} for $p>1$.  Fix $p\in(1,q)$.
Then, in view of H\"older's inequality,
\be\label{eqn:holder-1+2}
\begin{aligned}
\dashint_U \big|\psi_U^{\#}\big|^q\,\dd\bx
&\le
\Big(\big\|\psi_U^{\#}\big\|_{\infty,U}\Big)^{q-p}
\dashint_U \big|\psi_U^{\#}\big|^p\,\dd\bx,\\[2pt]
\Big|\,\dashint_U \psi \,\dd\bx \Big|^q
&\le
\Big|\,\dashint_U \psi \,\dd\bx \Big|^{q-p}
\dashint_U |\psi|^p \,\dd\bx.
\end{aligned}
\ee
\n Also, Lemma~\ref{lem:HL-B-domains} together with
\eqref{eqn:pointwise-bound}$_2$ yield a constant $H=H(p)>0$,
which does not depend on $\psi$, such that
\be\label{eqn:HL-again}
\dashint_U \big|\psi_U^{\#}\big|^p\,\dd\bx
\le
(2H)^p\dashint_U |\psi|^p\,\dd\bx.
\ee
%
%
\n If we now combine \eqref{eqn:Fefferman-Stein-two},
\eqref{eqn:holder-1+2}, and \eqref{eqn:HL-again} we find, with the aid of
\eqref{eqn:sharp=BMO} and the standard inequality
$|a|^t+|b|^t\le2(|a|+|b|)^t$ ($t>0$), that
\be\label{eqn:final-John}
\dashint_U |\psi|^q\,\dd\bx
\le
2FK\bigg(\tsn{\psi}_{\BMO(U)}
+ \Big|\,\dashint_U \psi \,\dd\bx \Big|\bigg)^{q-p}
\dashint_U |\psi|^p\,\dd\bx
\ee
\n with $K:=\max\{1,(2H)^p\}$. The desired result, \eqref{eqn:RH},
then follows from \eqref{eqn:BMO-norm} and \eqref{eqn:final-John}.
Once again, the scale invariance of $J_2$ is clear from
\eqref{eqn:final-John}, since the average value of any
function is scale invariant.

Finally, to obtain \eqref{eqn:RH} with $p=1$, we recall the
standard interpolation inequality
(see, e.g., \cite[p.~27]{AF03}),
for all $\psi\in L^1(U)\cap L^q(U)$ and $p\in(1,q)$,
\be\label{eqn:interp}
||\psi||_{p,U} \le
||\psi||_{1,U}^\theta
||\psi||_{q,U}^{1-\theta},
\qquad
\theta=\frac{\frac1p-\frac1q}{1-\frac1q}.
\ee
\n Inequality \eqref{eqn:RH}
with $p=1$ is then a consequence of
\eqref{eqn:RH} with $p>1$ and \eqref{eqn:interp}.
\end{proof}


\section{A Problem from the Calculus of
Variations}\label{sec:NLP-CV}

In this section we consider an energy minimization problem arising in the
Calculus of Variations.  We use the results in the previous section to
determine conditions upon a solution of the corresponding
Euler-Lagrange equations which imply that the solution is a local
minimizer of the energy in the $\BMO\cap\,L^1$-topology.

\subsection{Further Preliminaries}\label{sec:FP}

We denote the usual inner product of $\ba,\bb\in\R^d$ ($d=n$ or $d=N$)
by $\ba\cdot\bb$.
The norm of $\ba\in\R^d$ is then defined by $|\ba|:=\sqrt{\ba\cdot\ba}$.
We shall write
\be\label{eqn:M-norm}
|\bA|^2:=\sum_{i=1}^N\sum_{j=1}^n \big|A_{ij}\big|^2,
\ee
\n for the norm of $\bA \in\MN$ (the $N$ by $n$ matrices). Here
$A_{ij}$ denotes the component of $\bA$ from the $i$-th row and
the $j$-th column.

\emph{We fix a Lipschitz domain} $\Omega\subset\R^n$, $n\ge2$, with
boundary $\partial\Omega$. For $1\le p\le \infty$,
$W^{1,p}(\Omega;\R^N)$ will denote the usual
Sobolev space of (Lebesgue) measurable
(vector-valued) functions $\bu\in L^p(\Omega;\R^N)$ whose
distributional gradient $\grad\bu$ is also contained in $L^p$.
If $\phi\in W^{1,p}(\Omega)$ we shall denote its  $W^{1,p}$-norm by
\[
\begin{aligned}
||\phi||_{W^{1,p}(\Omega)}
&:=
\Big(||\phi||^p_{p,\Omega}
+||\grad\phi||^p_{p,\Omega}\Big)^{\!1/p},
\quad
1\le p <\infty,\\[2pt]
||\phi||_{W^{1,\infty}(\Omega)}
&
:= \max\{||\phi||_{\infty,\Omega}, ||\grad\phi||_{\infty,\Omega}\},
\quad
 p =\infty.
\end{aligned}
\]
\n For any $V\subset\R^n$ we denote the closure of $V$ by $\overline V$.


\subsection{An Integrand, the Energy, and the Euler-Lagrange
Equations}\label{sec:I-ELE}

We take
\[
\partial \Omega = \overline{\sD} \cup \overline{\sS}\quad
\text{with $\sD$ and $\sS$ relatively open and }
\sD\cap\sS=\varnothing.
\]
\n If $\sD\ne\varnothing$ we assume that a Lipschitz-continuous
function $\bd:\sD\to\R^N$
is prescribed.  If $\sS \ne\varnothing$ we assume that a function
$\bs\in L^2(\sS;\R^N)$  is prescribed.  We also suppose that a function
$\bb\in L^2(\Omega;\R^N)$ is prescribed.  In addition, we fix
a nonempty, open set $\cO\subset \MN$.

\begin{hypothesis}\label{def:W}  We suppose that we are
given an \emph{integrand}
$W:\overline{\Omega}\times\cO\to\R$
that satisfies:
\begin{enumerate}[topsep=-2pt]
\item $\bF\mapsto W(\bx,\bF) \in C^3(\cO)$, for $a.e.~\bx\in\Omega$;
\item $(\bx,\bF)\mapsto\DD^k W(\bx,\bF)$, $k=0,1,2,3$, are each
(Lebesgue) measurable on their common domain
$\Omega\times\cO$; and
\item  $(\bx,\bF)\mapsto\DD^k W(\bx,\bF)$, $k=0,1,2,3$,
are each bounded on
$\overline{\Omega}\times K$ for every compact $K\subset \cO$.
\end{enumerate}
\n Here, and in the sequel,
\[
\DD^0 W(\bx,\bF):=W(\bx,\bF), \qquad
\DD^k W(\bx,\bF)
:=
\frac{\partial^k}{\partial\bF^k}W(\bx,\bF)
\]
\n denotes $k$-th derivative of $\bF\mapsto W(\cdot,\bF)$.
\n Note that, for almost
every $\bx\in\overline{\Omega}$ and every $\bF\in\cO$,
\[
\begin{gathered}
\DD W(\bx,\bF):\MN\to\R,\qquad
\DD^2 W(\bx,\bF):\MN\times\MN\to\R
\end{gathered}
\]
\n can be viewed as a linear and a bilinear form,
respectively.
\end{hypothesis}


We denote the set of \emph{Admissible Mappings} by\footnote{Since
$\Om$ is a Lipschitz domain, each $\bu\in\AM$ has a representative
that is Lipschitz continuous.}
\[
\AM:=\{\bu\in W^{1,\infty}(\Omega;\R^N): \grad\bu\in \cO
\text{ and $\ \bu=\bd$ on $\sD$ or
$\langle\bu\rangle_\Omega=\mathbf{0}\, $ if $\, \sD=\varnothing$}\},
\]
\n where $\grad\bu\in \cO$ signifies that
$\grad\bu(\bx) \in \cO$ for $a.e.~\bx\in \Omega$.
The \emph{energy} of $\bu\in\AM$ is defined by
\be\label{eqn:energy}
\E(\bu):=\int_\Omega \big[W\big(\bx,\grad\bu(\bx)\big)
-\bb(\bx)\cdot\bu(\bx)\big]\,\dd\bx
-\int_{\sS} \bs(\bx)\cdot\bu(\bx)\,\dd\sH^{n-1}_\bx,
\ee
where $\sH^{k}$ denotes $k$-dimensional Hausdorff
measure.  We shall assume that we are given a
$\bu\e\in\AM$ that is a weak
solution of the \emph{Euler-Lagrange equations}
corresponding to \eqref{eqn:energy}, i.e.,
\be\label{eqn:ES}
0=\int_\Omega\big[
\DD W\big(\bx,\grad\bu\e(\bx)\big)[\grad\bw(\bx)]
-\bb(\bx)\cdot\bw(\bx)\big]\,\dd\bx
-
\int_\sS \bs(\bx)\cdot\bw(\bx)\,\dd \sH^{n-1}_\bx
\ee
\n for all \emph{variations} $\bw\in\Var$, where
\[
\Var:=\{\bw\in W^{1,2}(\Omega;\R^N):
\bw=\mathbf{0} \text{ on $\sD$ \ or \
$\langle\bw\rangle_\Omega=\mathbf{0}$ if $\sD=\varnothing$}\}.
\]
\n If $\sS=\varnothing$ then  $\bu\e$ is a solution of the
\emph{Dirichlet} problem. If $\sD=\varnothing$ then  $\bu\e$
is a solution of the
\emph{Neumann} problem. Otherwise, $\bu\e$ is a solution of the
\emph{mixed} problem.
For future reference we note that, for the Dirichlet problem,
the divergence theorem implies that, for all $\bw\in\Var$,
\be\label{eqn:D-D=0}
\int_\Omega \grad \bw(\bx)\,\dd\bx = \mathbf{0}.
\ee


We are interested in the local minimality (in an appropriate
topology) of solutions of  \eqref{eqn:ES}.
For future use we note that, for every $\bu,\bv\in\AM$,
\eqref{eqn:energy} gives us
\[
\E(\bv)-\E(\bu)
= \int_\Omega
\big[W\big(\grad\bv\big)-W\big(\grad\bu\big)
-\bb\cdot\bw\big]\dd\bx
-
\int_\sS \bs\cdot\bw\,\dd\sH^{n-1}_\bx,
\]
\n where $\bw:=\bv-\bu\in W^{1,\infty}(\Omega;\R^N) \cap \Var$.
It follows that, when $\bu\e\in\AM$ is a solution of
the Euler-Lagrange equations, \eqref{eqn:ES},
we have the identity, for every $\bv\in\AM$,
\be\label{eqn:Identity-ES}
\E(\bv)-\E(\bu\e)
= \int_\Omega
\Big(W\big(\bx,\grad\bv(\bx)\big)-W\big(\bx,\grad\bu\e(\bx)\big)
-\DD W\big(\bx,\grad\bu\e(\bx)\big)[\grad\bw(\bx)]\Big)\dd\bx.
\ee


For future reference we note that the second variation of the energy is
continuous in a certain ``direction'' in the
$\BMO\cap\,L^1$-topology.

\begin{lemma}\label{lem:extra}  Let $W$ satisfy (1)--(3) of
Hypothesis~\ref{def:W}.  Suppose that $\bu\in\AM$ satisfies,
for some $\hat{k}>0$ and
all $\bz\in\Var$,
\be\label{eqn:small+SV-in-Lem}
\begin{gathered}
\int_\Omega \DD^2 W\big(\bx,\grad\bu(\bx)\big)
\big[\grad\bz(\bx), \grad\bz(\bx)\big]\,\dd\bx
\ge
8\hat{k}\int_\Omega |\grad\bz(\bx)|^2\dd\bx,\\[4pt]
\grad\bu(\bx)\in\sB \text{ for $a.e.~\bx\in \Omega$},
\end{gathered}
\ee
\n where $\sB$ is a nonempty, bounded, open set
with $\overline{\sB}\subset\cO\subset\MN$.
Then there exists
an $\varepsilon>0$ such that
any $\bv\in\AM$ that satisfies, for $a.e.~\bx\in \Omega$,
\be\label{eqn:small+B-in-Lem}
\grad\bv(\bx)\in\sB,
\qquad
\tsn{\grad\bv-\grad\bu}_{\BMO(\Omega)}<\varepsilon,
\qquad
\Big|\dashint_\Omega (\grad\bv-\grad\bu)\,\dd\bx\Big|<\varepsilon
\ee
\n will also satisfy
\be\label{eqn:small+SV-in-Lem-2}
\int_\Omega \DD^2 W(\bx,\grad\bv(\bx))
\big[\grad\bw(\bx), \grad\bw(\bx)\big]\,\dd\bx
\ge
4\hat{k}\int_\Omega \big|\grad\bw(\bx)\big|^2\,\dd\bx,
\quad \bw:=\bv-\bu.
\ee
\end{lemma}


Here and in the sequel, we use the notation
$\tsn{\grad\bv-\grad\bu}_{\BMO(\Omega)}$
to denote the $\BMO$-seminorm of the tensor $\grad\bv-\grad\bu$.
The definition is precisely as in \eqref{eqn:BMO-V-norm} and
\eqref{eqn:sharp=BMO}, except one makes a slight modification to the
former equation.  In particular, for the sharp function
\eqref{eqn:BMO-V-norm} one has the tensor in place of $\psi$ and the
Euclidean norm in place of the absolute value in the integral.

\begin{proof}
For clarity of exposition, we suppress the variable $\bx$.
Let $\bu\in\AM$ satisfy \eqref{eqn:small+SV-in-Lem} for
all $\bz\in\Var$.  Suppose that
$\bv\in\AM$ satisfies \eqref{eqn:small+B-in-Lem}
for some $\varepsilon>0$ to be determined later and
define $\bw:=\bv-\bu$.
Then, Lemma~\ref{lem:taylor-III-AA} with $\bG=\grad\bv$, $\bF=\grad\bu$,
and $\bL=\bG-\bF=\grad\bw$  yields a constant
$\hat{c}=\hat{c}(\sB)>0$ such that,
for $a.e.~\bx\in \Omega$,
\be\label{eqn:taylor-app-in-Lem}
\DD^2 W(\grad\bv)[\grad\bw, \grad\bw]\ge \DD^2 W(\grad\bu)
[\grad\bw, \grad\bw] - \hat{c}|\grad\bw|^3.
\ee

 If we now integrate
\eqref{eqn:taylor-app-in-Lem} over $\Omega$ and make use of the uniform
positivity of the second variation, \eqref{eqn:small+SV-in-Lem}$_1$,
we find that
\be\label{eqn:E-taylor-in-Lem}
\int_\Omega \DD^2 W(\grad\bv)[\grad\bw, \grad\bw]\,\dd\bx
\ge
8\hat{k}\int_\Omega|\grad\bw|^2\,\dd\bx
-\hat{c}\int_\Omega|\grad\bw|^3\,\dd\bx.
\ee
\n We next note that inequality \eqref{eqn:RH} (with $q=3$ and $p=2$) of
Theorem~\ref{thm:main-2} yields a $J>0$ such
that, for the given $\bu$ and $\bv$ that satisfy
\eqref{eqn:small+B-in-Lem}$_{2,3}$ and every $i\in \{1,\ldots, N\}$
and $j \in \{1,\ldots,n\}$,
\be\label{eqn:quadratic>cubic-in-Lem}
2\varepsilon J^3\int_\Omega
\Big|\frac{\partial w_i}{\partial x_j}\Big|^2\dd\bx
\ge
\int_\Omega
\Big|\frac{\partial w_i}{\partial x_j}\Big|^3\dd\bx,
\qquad \bw:=\bv-\bu.
\ee
\n Thus one deduces \eqref{eqn:small+SV-in-Lem-2} as a consequence of
\eqref{eqn:E-taylor-in-Lem} and \eqref{eqn:quadratic>cubic-in-Lem}
when $\varepsilon$ is sufficiently small.
\end{proof}


\subsection{Implications of the Positivity of the Second
Variation}\label{sec:IPSV}

In this subsection we show that any admissible mapping $\bv$
with gradient sufficiently close, in $\BMO\cap\,L^1$, to the gradient of
a Lipschitz solution of the Euler-Lagrange equations
whose second variation is
uniformly positive, will have strictly greater energy than the solution.
In addition, it will follow that such a $\bv$ cannot itself satisfy
the Euler-Lagrange equations.


\begin{theorem}\label{thm:SVP=LMBMO} Let $W$ satisfy (1)--(3) of
Hypothesis~\ref{def:W}.  Suppose that $\bu\e\in\AM$ is a
 weak solution of the Dirichlet, Neumann, or mixed problem, i.e.,
\eqref{eqn:ES},
that satisfies, for some $\hat{k}>0$ and all $\bz\in\Var$,
\be\label{eqn:small+SV}
\begin{gathered}
\int_\Omega \DD^2 W\big(\bx,\grad\bu\e(\bx)\big)
\big[\grad\bz(\bx), \grad\bz(\bx)\big]\,\dd\bx
\ge
8\hat{k}\int_\Omega |\grad\bz(\bx)|^2\dd\bx,\\[4pt]
\grad\bu\e(\bx)\in\sB \text{ for $a.e.~\bx\in \Omega$},
\end{gathered}
\ee
\n where $\sB\ne\varnothing$ is a bounded open set
with $\overline{\sB}\subset\cO\subset\MN$.
Then there exists
a $\delta=\delta(\sB)>0$ such that
any $\bv\in\AM$ that satisfies, for $a.e.~\bx\in \Omega$,
\be\label{eqn:small+B}
\grad\bv(\bx)\in\sB,
\qquad
\tsn{\grad\bv-\grad\bu\e}_{\BMO(\Omega)}<\delta,
\qquad
\Big|\dashint_\Omega (\grad\bv-\grad\bu\e)\,\dd\bx\Big|<\delta
\ee
\n will also satisfy
\be\label{eqn:E-taylor-repeat-2}
\E(\bv)\ge \E(\bu\e)
+\hat{k}\int_\Omega|\grad\bv-\grad\bu\e|^2\dd\bx.
\ee
\n  In particular, $\bv\not\equiv\bu\e$ will have strictly greater
energy than $\bu\e$.  Moreover,  $\bv$ cannot be a solution of
the Euler-Lagrange equations, \eqref{eqn:ES}.
\end{theorem}


\begin{remark} (1).~For the Dirichlet problem, \eqref{eqn:D-D=0}
shows that the integral in \eqref{eqn:small+B}$_3$ is equal to zero;
consequently, \eqref{eqn:small+B}$_3$ is trivially satisfied for any
$\delta>0$.
(2).~Since we have assumed that $\bu\e\in W^{1,\infty}(\Omega;\R^n)$,
sets  $\sB\subset \MN$ that satisfy
\eqref{eqn:small+SV}$_2$ do exist, e.g.,
\[
\sB:=\sB(||\grad\bu\e||_\infty)
=
\{\bF\in \MN: |\bF|< 1+||\grad\bu\e||_{\infty,\Omega}\}.
\]
\n However, the integrand $W:\overline{\Omega}\times\cO\to\R$
need not be
defined on all of $\overline{\Omega}\times\sB(||\grad\bu\e||_\infty)$.
For example, in Nonlinear Elasticity (see Section~\ref{sec:CR}) one usually
assumes
that\footnote{Here $\det\bF$ denotes
the determinant of $\bF\in\Mn$.}
\[
\cO = \{ \bF\in \Mn:  \det \bF >0\}
\]
\n in which case $\mathbf{0}\not\in\cO$ and hence
$\sB(||\grad\bu\e||_\infty)\not\subset \cO$.
\end{remark}


\begin{remark}\label{rem:KTCC}
Kristensen \& Taheri~\cite[Section~6]{KT03} and
Campos~Cordero~\cite[Section~4]{Ca17} have each obtained a result that
is analogous to Theorem~\ref{thm:SVP=LMBMO} for
Dirichlet boundary data.  In particular, they show that, under weaker
smoothness hypotheses than used here ($\bF\mapsto W(\bF)\in C^2(\MN)$ and
$(\bx,\bF)\mapsto W(\bx,\bF)\in C^2(\overline{\Omega}\times\MN)$,
respectively), one has uniqueness in the regime
where the extension by zero of $\bH(\bx):=\grad\bv(\bx)-\grad\bu\e(\bx)$
is sufficiently small as an element of $\BMO(\R^n)$.
The extension of our result to $C^2$ integrands appears to depend on a
particular generalization of the Fefferman-Stein inequality to bounded
domains: more precisely,  a version of Theorem~\ref{thm:main-1} for
certain Orlicz spaces.  The proofs of
Lemma~6.2 in \cite{KT03} and Lemmas~4.6 and 4.7 in \cite{Ca17} modify
the Fefferman-Stein inequality on all of $\R^n$ by introducing the
modulus of continuity, $\omega$, of $\DD^2 W$ in Taylor's theorem and
then making use of $t\mapsto t^2\omega(t)$ as an $N$-function
(see, e.g., \cite{AF03}).
Such an extension for \emph{cubes} has been obtained by
Verde \& Zecca~\cite[Theorem~2.1]{VZ08}, however, we are not aware of
any corresponding proof for Lipschitz (or John) domains.
\end{remark}


\begin{proof}[Proof of Theorem~\ref{thm:SVP=LMBMO}] For clarity of
exposition, we suppress the variable $\bx$.
Let $\bu\e\in\AM$ be a solution of the Euler-Lagrange equations,
\eqref{eqn:ES}, that satisfies \eqref{eqn:small+SV}
for all $\bz\in\Var$.  Suppose that
$\bv\in\AM$ satisfies \eqref{eqn:small+B}
for some $\delta>0$ to be determined later and define
$\bw:=\bv-\bu\e\in\Var$. Then,
Lemma~\ref{lem:taylor-III} with $\bG=\grad\bv$, $\bF=\grad\bu\e$, and
$\bH=\bG-\bF=\grad\bw$,
yields a constant $c=c(\sB)>0$ such that,
for $a.e.~\bx\in \Omega$,
\be\label{eqn:taylor-app}
W(\grad\bv)\ge W(\grad\bu\e)
+ \DD W(\grad\bu\e)[\grad\bw]
+\tfrac12\DD^2 W(\grad\bu\e)[\grad\bw,\grad\bw]
-c|\grad\bw|^3.
\ee
\n If we now integrate
\eqref{eqn:taylor-app} over $\Omega$ and make use of the uniform
positivity of the second variation, \eqref{eqn:small+SV}$_1$, we find,
with the aid of \eqref{eqn:Identity-ES} (which is a
consequence of the fact that $\bu\e$ satisfies
the Euler-Lagrange equations \eqref{eqn:ES}), that
\be\label{eqn:E-taylor}
\E(\bv)\ge \E(\bu\e)
+2\hat{k}\int_\Omega|\grad\bw|^2\,\dd\bx
-c\int_\Omega|\grad\bw|^3\,\dd\bx.
\ee


We next note that inequality \eqref{eqn:RH} (with $q=3$ and $p=2$) of
Theorem~\ref{thm:main-2} yields a $J>0$ such
that, for the given $\bu\e$ and $\bv$ that satisfy
\eqref{eqn:small+B}$_{2,3}$ and every $i\in \{1,\ldots, N\}$
and $j \in \{1,\ldots,n\}$,
\be\label{eqn:quadratic>cubic-in-Lem-2}
2\delta J^3\int_\Omega
\Big|\frac{\partial w_i}{\partial x_j}\Big|^2\dd\bx
\ge
\int_\Omega
\Big|\frac{\partial w_i}{\partial x_j}\Big|^3\dd\bx,
\qquad \bw:=\bv-\bu\e.
\ee
\n Again one finds that \eqref{eqn:E-taylor-repeat-2} follows from
\eqref{eqn:E-taylor} and \eqref{eqn:quadratic>cubic-in-Lem-2} when
$\delta$ is sufficiently small.

Now, suppose that $\E(\bv)=\E(\bu\e)$.  Then
\eqref{eqn:E-taylor-repeat-2} yields $\grad\bv=\grad\bu\e$ in
$\Omega$ and hence, since $\Omega$ is open and connected,
$\bv=\bu\e+\ba$ for some $\ba\in\R^N$.  However,
$\bw=\bv-\bu\e\in\Var$ and so either
$\bv=\bu\e$ on $\sD$ or $\langle\bw\rangle_\Omega=\mathbf{0}$,
both of which force $\ba=\mathbf{0}$.  Thus,
$\E(\bv)=\E(\bu\e)$ implies $\bv\equiv\bu\e$.

Finally, we note that Lemma~\ref{lem:extra} shows that, if
$\delta\in(0,\varepsilon)$, then the second variation of the energy
is uniformly positive in the direction $\bv-\bu\e$ at $\bv$, that is,
$\bv$ satisfies
\eqref{eqn:small+SV} with $\bu\e$ replaced by $\bv$ and
$\bz=\bv-\bu\e$.  Now, suppose for the sake of contradiction
that $\bv\not\equiv\bu\e$ is also a solution of \eqref{eqn:ES}. Then,
the above argument, with $\bu\e$ replaced by $\bv$ and $\bv$
replaced by $\bu\e$, shows that $\E(\bu\e)>\E(\bv)$, which contradicts
$\E(\bv)>\E(\bu\e)$.  Thus, two distinct solutions of \eqref{eqn:ES},
both of which satisfy  
\eqref{eqn:small+B}, is not possible.
\end{proof}


\part*{Part II: Rotations, Sobolev Mappings, and Nonlinear
Elasticity}\label{part:R-SM-NE}

\section{Rotations, Geometric Rigidity, and Sobolev
Mappings}\label{sec:RGRSM}

In this section we consider the set of $n$-dimensional rotations with
an interest in a comparison of
the distance of a Sobolev mapping from this set to the
distance the mapping has from a single rotation.

\subsection{Additional Preliminaries}\label{sec:prelim-2}
We shall write $\bH\!:\!\bK:=\tr(\bH\bK^\rmT)$ for the inner product of
$\bH,\bK \in\Mn$, where $\tr$
denotes the trace and $\bK^\rmT$ denotes the transpose of $\bK$.
The norm of $\bH \in\Mn$, which is defined by \eqref{eqn:M-norm},
is then equal to
$\sqrt{\bH\!:\!\bH\,}$.
We shall denote the set of $n$-dimensional \emph{rotations} by $\SOn$;
thus, every $\bR\in\SOn$ satisfies $\bR^\rmT\bR=\bR\bR^\rmT=\bI$
and $\det\bR=1$, where $\bI\in\Mn$ denotes the identity matrix.
If $\bV\in\Mn$  is invertible, we use the notation
$\bV^{\mi1}$ to denote its inverse, viz.,
$\bV\bV^{\mi1}=\bV^{\mi1}\bV=\bI$.


We use the notation $\wedge$ to denote the exterior
(``wedge'') product
(see, e.g., \cite[Chapter~1]{Fe69}, \cite[Chapter~9]{IM01},
or \cite[Chapter~4]{Sp65}). For $n\ge3$ we shall identify the
space $\Lambda_{n-1}\R^{n-1}$, of alternating $n-1$ tensors
on $\R^n$, with $\R^n$ itself by means
of the mapping
\[
(\ba_1,\ba_2,\ldots,\ba_{n-1})
\mapsto
\ba_1\wedge\ba_2\wedge\ldots\wedge\ba_{n-1}.
\]
\n We note that this mapping is multilinear, alternating, and
satisfies
\be\label{eqn:wedge-1}
\bee_1\wedge\bee_2\wedge\ldots\wedge\bee_{n-1} = \bee_n
\ee
\n when $\bee_1,\bee_2,\ldots,\bee_n$ is
\emph{any orthonormal basis with the standard orientation for}
$\R^n$.  We shall also make use of the identities,
for all rotations $\bQ\in\SOn$,
\be\label{eqn:wedge-2}
\begin{gathered}
\bQ\bee_1\wedge\bQ\bee_2\wedge\ldots\wedge\bQ\bee_{n-1}
= \bQ(\bee_1\wedge\bee_2\wedge\ldots\wedge\bee_{n-1})=\bQ\bee_n,\\
|\ba_1\wedge\ba_2\wedge\ldots\wedge\ba_{n-1}|
\le X \prod_{k=1}^{n-1}|\ba_k|,
\end{gathered}
\ee
\n for all $\ba_k\in\R^n$, where $X=X(n)>0$ is a
constant that depends only on the dimension $n$.


\begin{remark} (1).~When $n=3$ the usual cross product can be
substituted for the wedge product;  also $X(3)=1$.
(2).~Equation \eqref{eqn:wedge-2}$_1$ follows
from \eqref{eqn:wedge-1}; the exterior product of the
first $n-1$ vectors in any standardly oriented orthonormal
basis yields the unique unit vector, with the proper orientation,
that is perpendicular to each of the other vectors.  For
\eqref{eqn:wedge-2}$_2$ see, e.g., \cite[p.~220]{IM01}.
\end{remark}


\subsection{The Geometric-Rigidity Theory of Friesecke,
James, \& M\"uller}\label{sec:GRTFJM}  In Theorem~3.1 in
\cite{FJM02} the authors have shown that, given a Sobolev mapping $\bu$,
there exists a rotation $\bR_\bu$ such that the distance
from $\grad\bu$ to $\bR_\bu$ is, up to a
multiplicative constant which does not depend on $\bu$, a lower
bound for the distance from $\grad\bu$ to the
set of  $n$-dimensional rotations.
Their measure of distance from the set of rotations
is the $L^2$-norm of the functional
\[
\dist\!\big(\grad\bv(\bx),\SOn\big)
:=
\min_{\bQ\in\SOn} |\grad\bv(\bx)-\bQ|.
\]
\n However, as noted by Conti \& Schweizer~\cite[p.~854]{CS06}, $L^2$
can be replaced by $L^p$ for any $p\in(1,\infty)$.


Before we state the Geometric-Rigidity result of interest
in this manuscript,
we first note that, when the Jacobian of a mapping is strictly positive,
the distance to the set
of rotations can be expressed in an alternative form.
We give a
proof for the convenience of the reader.


\begin{lemma}\label{lem:alt-dist}  Let $\bF\in\Mn$ with polar
 decomposition $\bF=\bR\bU$ satisfy $\det\bF>0$.
Then
\[
\dist\!\big(\bF,\SOn\big)=\big|\sqrt{\bF^\rmT\bF\,}-\bI\big|=|\bU-\bI|.
\]
\end{lemma}
\begin{proof}  Recall that (see, e.g., \cite[Chapter~I]{Gu81} or
\cite[Section~3.2]{Ci1988}) $\bF\in\Mn$ with
$\det\bF>0$ has a unique polar decomposition $\bF=\bR\bU$, where
$\bU:=\sqrt{\bF^\rmT\bF}$ is symmetric and strictly positive definite
and $\bR:=\bF\bU^{\mi1}\in \SOn$.   Then, for any $\bQ\in\SOn$,
\be\label{eqn:square}
|\bF-\bQ|^2 = |\bF|^2-2\bF:\bQ+n=|\bU|^2-2\bU:\bR^\rmT\bQ+n.
\ee
\n Next, by the spectral theorem,
\be\label{eqn:spectral}
\bU:\bR^\rmT\bQ
= \sum_{k=1}^n \lambda_k\big[\bff_k\otimes\bff_k\big] :\bR^\rmT\bQ
= \sum_{k=1}^n \lambda_k\bff_k\cdot\bR^\rmT\bQ\bff_k,
\ee
\n where $\lambda_k>0$ and $\{\bff_k:k=1,2,\ldots,n\}$ is
an orthonormal basis for $\R^n$.
Consequently, in view of \eqref{eqn:square} and \eqref{eqn:spectral}
the minimum of $|\bF-\bQ|$ will occur when each of the quantities
$\bff_k\cdot\bR^\rmT\bQ\bff_k$ is maximized, that is, when
$\bR^\rmT\bQ=\bI$.  Therefore,
\[
\dist\!\big(\bF,\SOn\big)
:=\min_{\bQ\in\SOn} |\bF-\bQ| = |\bR\bU-\bR|= |\bU-\bI|,
\]
\n as claimed.
\end{proof}

We now state the result that we shall utilize.


\begin{proposition}\label{prop:GR}
\emph{({\cite[Section~3]{FJM02}) and \cite[Section~2.4]{CS06}})}
Let $1<p<\infty$.  Suppose that $\Om\subset\R^n$, $n\ge2$, is a
bounded Lipschitz domain.  Then there exists a
constant $C=C(p,\Om)$ with the following property:
For each $\bv\in W^{1,p}(\Om;\R^n)$
there is an associated rotation $\bR=\bR(p,\bv,\Om)\in \SOn$ such that
\be\label{eqn:GR}
\dashint_\Om |\grad\bv(\bx)-\bR|^p\,\dd\bx
\le
 C^p \dashint_\Om \Big[\dist\!\big(\grad\bv(\bx),\SOn\big)\Big]^p\dd\bx.
\ee
\n Moreover, \eqref{eqn:GR} is scale invariant, i.e.,
$C(p,\lambda \Om+\ba)=C(p,\Om)$ for all $\lambda>0$ and $\ba\in\R^n$.
In addition, there exists a constant $M=M(n)>0$ such that, for all
$\bv\in W^{1,\infty}(\Om;\R^n)$,
\be\label{eqn:small-dist-&-small_BMO}
\tsn{\grad\bv}_{\BMO(\Om)}
\le M \big\|\dist\!\big(\grad\bv,\SOn\big)\big\|_{\infty,\Om}.
\ee
\end{proposition}


\begin{remark}\label{rem:GR+1} (1).~When $p=1$ or $p=\infty$
the estimate
corresponding to \eqref{eqn:GR} is \emph{not} valid.  See
John~\cite[pp.~393--394]{Jo61} for a counterexample when $p=\infty$.
(2).~When $p=1$ Conti \& Schweizer~\cite[p.~853]{CS06} obtained a
so-called \emph{weak-type} estimate as well as an estimate
where the integral on the right-hand side of \eqref{eqn:GR}, which
we here denote by $\rho$, is replaced by $\rho\max\{-\ln \rho, 1\}$.
(3).~The result in \cite{FJM02} corresponding to
\eqref{eqn:small-dist-&-small_BMO} differs slightly.  However,
the above version is a direct consequence of \eqref{eqn:GR}, H\"olders
inequality, the scale invariance of $C$, and the definition
of $\BMO(\Om)$.
(4).~Inequalities \eqref{eqn:GR} and \eqref{eqn:small-dist-&-small_BMO}
were first obtained by
John~\cite{Jo61,Jo72-2} when $\Om$ is a cube, $\bv$ is $C^1$, and the
norm on the right-hand side of \eqref{eqn:small-dist-&-small_BMO} is
sufficiently small.
(5).~Conti, Dolzmann, \& M\"uller~\cite[Section~4]{CDM14} have obtained
a version of
\eqref{eqn:GR} for the Lorentz spaces $L^{p,q}(\Om)$, $p\in(1,\infty)$,
$q\in[1,\infty]$.
(6).~Ciarlet \& Mardare~\cite{CM15} have
obtained a version of
\eqref{eqn:GR} (but not \eqref{eqn:small-dist-&-small_BMO}) that
involves two mappings.  See Remark~\ref{rem:CM-15} in this
manuscript for a brief description of one of their results.
(7).~See, also, ~Re\v setnjak~\cite{Re67},
Benyamini \& Lindenstrauss~\cite[Chapter~14]{BL00}, and
Fefferman, Damelin, \& Glover~\cite{DFG12}.
\end{remark}

\begin{remark}\label{rem:GR+2}  The distance of the mapping $\bv$
to the closest rigid
mapping, $\br(\bx)=\bR\bx+\ba$, is also of interest. Such estimates
follow from \eqref{eqn:GR} upon application of a standard embedding
theorem or the Poincar\'{e} inequality.
John~\cite{Jo61,Jo72-2} obtained such a result for cubes when the
$L^\infty$-norm on the right-hand side of
\eqref{eqn:small-dist-&-small_BMO} is
sufficiently small.   Kohn~\cite{Ko82} proved a similar result for
Lipschitz domains when the mappings were bi-Lipschitz, but without
the need for an $L^\infty$ bound.  He also obtained a bound similar
to \eqref{eqn:GR} for bi-Lipschitz mappings.
\end{remark}

\begin{remark}\label{rem:GR+3} If $\bG:=\langle\grad\bv\rangle_\Om$
satisfies $\det\bG>0$,  a short
computation (see the proof of Lemma~\ref{lem:alt-dist}) shows that
\[
\min_{\bQ\in\SOn}\int_\Om |\grad\bv(\bx)-\bQ|^2\,\dd\bx
\]
\n is achieved when $\bQ:=\bG\bV^{\mi1}$,
where $\bV=\sqrt{\bG^\rmT\bG}$,
i.e., $\bG$ has polar decomposition $\bG=\bQ\bV$. This was
first noticed by John~\cite{Jo61,Jo72-2}.
\end{remark}


\subsection{Sobolev Mappings and Rotations}\label{sec:RSM}

In this subsection we show that the imposition of a Dirichlet boundary
condition on a nonempty, relatively open subset of the boundary yields a
relationship between Sobolev mappings and rotations.
Recall that $\Omega\subset\R^n$ is a fixed Lipschitz domain and suppose
that $\sD\subset \partial\Omega$ is a nonempty, relatively open set.

\begin{lemma}\label{prop:BC}  Fix $p\in(n,\infty)$. Then there
exists a constant $A=A(p,\Omega,\sD)>0$ such that every pair of
mappings $\bu^{(i)}\in W^{1,p}(\Omega;\R^n)$, $i=1,2$, that satisfies
$\bu^{(1)}(\bx)=\bu^{(2)}(\bx)$ for $\bx\in\sD$,
will also satisfy
\be\label{eqn:rotations-close}
\big|\bR^{(1)}-\bR^{(2)}\big|
<
A\Big(\big\|\grad\bu^{(1)}-\bR^{(1)}\big\|_{p,\Omega}
+
\big\|\grad\bu^{(2)}-\bR^{(2)}\big\|_{p,\Omega}\Big)
\ee
\n for every pair of rotations $\bR^{(i)}\in\SOn$, $i=1,2$.
\end{lemma}


Before we prove Lemma~\ref{prop:BC}, we first present an interesting
consequence.

\begin{proposition}\label{prop:BC+SR=close-in-L1} Fix $p\in(n,\infty)$.
Then there exists a constant $A^*=A^*(p,\Omega,\sD)>0$
such that every pair of mappings $\bu^{(i)}\in W^{1,p}(\Omega;\R^n)$,
$i=1,2$, that satisfies $\bu^{(1)}(\bx)=\bu^{(2)}(\bx)$ for $\bx\in\sD$,
will also satisfy
\be\label{eqn:BC+SR=close-in-L1}
\big\|\grad\bu^{(1)}-\grad\bu^{(2)}\big\|_{1,\Om}
\le
A^*
\Big(\big\|\dist\!\big(\grad\bu^{(1)},\SOn\big)\big\|_{p,\Om}
+
\big\|\dist\!\big(\grad\bu^{(2)},\SOn\big)\big\|_{p,\Om}\Big).
\ee
\end{proposition}


\begin{proof}  Fix $p>n$ and suppose that
$\bu^{(i)}\in W^{1,p}(\Omega;\R^n)$,
$i=1,2$.  Then,  in view of Proposition~\ref{prop:GR}, there exist
rotations $\bR^{(i)}\in\SOn$ that satisfy
\be\label{eqn:GR-again}
\big\|\grad\bu^{(i)}-\bR^{(i)}\big\|_{p,\Omega}
\le
 C \big\|\dist\!\big(\grad\bu^{(i)},\SOn\big)\big\|_{p,\Omega}
\ee
\n for some constant $C=C(p,\Om)$.  If we now add and subtract $\bR^{(1)}$
and $\bR^{(2)}$ from $\grad\bu^{(1)}-\grad\bu^{(2)}$ and take the
$L^1$-norm of the result we find, with the aid of the triangle
inequality, that
\be\label{eqn:triangle}
\big\|\grad\bu^{(1)}-\grad\bu^{(2)}\big\|_{1,\Om}
\le
|\Om|\big|\bR^{(1)}-\bR^{(2)}\big| +\sum_{i=1}^2
\big\|\grad\bu^{(i)}-\bR^{(i)}\big\|_{1,\Omega}.
\ee
\n The desired result, \eqref{eqn:BC+SR=close-in-L1}, now follows from
\eqref{eqn:triangle}, \eqref{eqn:GR-again}, Lemma~\ref{prop:BC},
and H\"older's inequality.
\end{proof}


\begin{proof}[Proof of Lemma~\ref{prop:BC}]
Given $\bu^{(i)}\in W^{1,p}(\Omega;\R^n)$ and $\bR^{(i)}\in\SOn$
define, for $i=1,2$,
\[
d_i:=||\grad\bu^{(i)}-\bR^{(i)}||_{p,\Omega}, \qquad
\ba^{(i)}
:=
\langle\bu^{(i)}-\bR^{(i)}\bx\rangle_\Omega.
\]
\n Then, the Poincar\'{e} inequality
(see, e.g., \cite[p.~361]{Le09} or \cite[p.~218]{LL01})
yields a constant $ P > 0$, which is independent of
$\bu^{(i)}$, $\bR^{(i)}$, and $\ba^{(i)}$, such that
\be\label{eqn:P}
 ||\bu^{(i)}-\bR^{(i)}\bx-\ba^{(i)}||_{W^{1,p}(\Omega)}
\le
P d_i.
\ee
\n Next, since $p>n$ we have the imbedding (see, e.g.,
\cite[Section~4.27]{AF03})
$W^{1,p}(\Omega) \to C^{0,\lambda}(\overline{\Omega})$,
i.e., there is a constant $M>0$ such that,
for every $\bx,\by\in\overline{\Omega}$ with $\bx\ne\by$,
\be\label{eqn:ME}
||\bv^{(i)}||_{\infty,\Omega} +
\frac{|\bv^{(i)}(\bx)-\bv^{(i)}(\by)|}{|\bx-\by|^\lambda}
\le M||\bv^{(i)}||_{W^{1,p}(\Omega)},
\ee
\n  where
$\bv^{(i)}(\bx):=\bu^{(i)}(\bx)-\bR^{(i)}\bx-\ba^{(i)}$.
Here $\lambda:=1-n/p$.
We now note that \eqref{eqn:P} together with \eqref{eqn:ME}
implies that, for all $\bx,\by\in\overline{\Omega}$,
\be\label{eqn:HC}
|\bv^{(i)}(\bx)-\bv^{(i)}(\by)|\le  M P d_i|\bx-\by|^\lambda.
\ee


Now, suppose that $\bx,\by\in\sD$; then
$\bu^{(1)}(\bx)=\bu^{(2)}(\bx)$ and
$\bu^{(1)}(\by)=\bu^{(2)}(\by)$.  Thus,
\be\label{eqn:both}
\begin{aligned}
\big(\bR^{(1)}-\bR^{(2)}\big)[\by-\bx]
&=
\bR^{(1)}[\by-\bx]-\bR^{(2)}[\by-\bx]\\
&+\big(\bu^{(1)}(\bx)-\bu^{(1)}(\by)\big)
-\big(\bu^{(2)}(\bx)-\bu^{(2)}(\by)\big).
\end{aligned}
\ee
\n Define $\bR:=[\bR^{(1)}]^\rmT\bR^{(2)}\in\SOn$ and note that,
for all $\bb\in\R^n$,
\be\label{eqn:del-1-2}
\big|\big(\bR^{(1)}-\bR^{(2)}\big)\bb\big|
=
\big|\bR^{(1)}\big(\bI-\big[\bR^{(1)}\big]^\rmT\bR^{(2)}\big)\bb\big|
=
|(\bI-\bR)\bb|.
\ee
\n Therefore, if we take the norm of \eqref{eqn:both}, the
triangle inequality together with \eqref{eqn:HC}, the definition of
the $\bv^{(i)}$, and  \eqref{eqn:del-1-2} yield
\be\label{eqn:final-a}
\big|(\bI-\bR)[\by-\bx]\big|\le M P d|\by-\bx|^\lambda\
\text{ for all $\bx,\by\in\sD$,}
\ee
\n where $d:=d_1+d_2$.


Next, $\partial \Omega$ is Lipschitz; thus, we can fix an
$\bx\oo\in\sD$ where $\partial \Omega$ has a unique outward unit normal
vector and tangent hyperplane.  Then,
\emph{with a change in coordinates,} let $\bx\oo=\mathbf{0}$ and suppose
that $\{\bee_1,\bee_2,\ldots,\bee_n\}$ is a basis for $\R^n$
(with the standard orientation) with $\bee_n$ the outward
unit normal at $\mathbf{0}$ and the tangent hyperplane,
$\sT\subset\R^n$, at $\mathbf{0}$
given as the span of $\{\bee_1,\bee_2,\ldots,\bee_{n-1}\}$.  Moreover,
since $\sD$ is relatively open and $\partial \Omega$ is Lipschitz,
there exists an open ball
$B=B(\mathbf{0},2r)\subset \R^{n-1}$ and a Lipschitz function
$\gamma:B\to\R$ such that $(\bz,\gamma(\bz))$ with $\bz\in B$
is a relatively open subset of $\sD$ and $\gamma(\mathbf{0})=0$.


For any $\bz\in\R^{n-1}$ that satisfies $|\bz|\le r$,
inequality \eqref{eqn:final-a} implies that
\be\label{eqn:bound-at-B}
|(\bI-\bR)\bt|
\le
|(\bI-\bR)\by_\gamma|
\le
M P d|\by_\gamma|^\lambda,  \qquad \by_\gamma:=(\bz, \gamma(\bz)),
\qquad
\bt=(\bz, 0)\in \sT.
\ee
\n Also, $\gamma$ is Lipschitz continuous; consequently,
 there exists a $L>0$ such that
(recall that $\gamma(\mathbf{0})=0$)
\be\label{eqn:Lipschitz}
|\gamma(\bz)|\le L |\bz|
\quad
\text{ and hence }
\quad
|\by_\gamma|
\le \sqrt{1 + L^2}\, |\bz|.
\ee
\n If we now combine \eqref{eqn:bound-at-B} and \eqref{eqn:Lipschitz}
we find that, for all $\bt\in\sT$ with $|\bt|\le r$,
\be\label{eqn:tan-small-1}
|(\bI-\bR)\bt|
\le
G d |\bz|^\lambda = G d |\bt|^\lambda,  \qquad \bt=(\bz, 0),
\ee
\n where $G=G(p,n,\Omega):= M P (1+L^2)^{\lambda/2}$.  In particular,
the choice $\bt=r\bee_k$, $k=1,2,3,\ldots,n-1$ in
\eqref{eqn:tan-small-1} yields
\be\label{eqn:tan-small-2}
|(\bI-\bR)\bee_k|\le Gd r^{\lambda-1}\ \text{ for $1\le k\le n-1$.}
\ee


Finally, we shall show that\footnote{Recall that in $2$-dimensions
all rotations commute.  Consider the rotation, $\bQ_{12}$, that
satisfies $\bQ_{12}\bee_1=\bee_2$. It follows that
$(\bI-\bR)\bee_2=(\bI-\bR)\bQ_{12}\bee_1=\bQ_{12}(\bI-\bR)\bee_1$ and
hence $|(\bI-\bR)\bee_2|=|\bQ_{12}(\bI-\bR)\bee_1|=|(\bI-\bR)\bee_1|$,
which, by \eqref{eqn:tan-small-2}, is bounded above
by $G d r^{\lambda-1}$.},
if $n\ge3$, then \eqref{eqn:tan-small-2}
is also satisfied when $k=n$ and $G$ is replaced by $(n-1) G X$,
where $X$ is the constant
from \eqref{eqn:wedge-2}$_2$.  This will imply that
(see \eqref{eqn:del-1-2} and \eqref{eqn:bound-at-B}$_2$)
\[
\big|\bR^{(1)}-\bR^{(2)}\big|
\le
\sqrt{n}\sup_{|\bee|=1} |(\bI-\bR)\bee|
\le
\sqrt{n}(n-1)X G(d_1+d_2) r^{\lambda-1},
\]
\n which is \eqref{eqn:rotations-close}
with $A=\sqrt{n}(n-1) M P X r^{\lambda-1}(1+L^2)^{\lambda/2}$.


In order to estimate $|\bR\bee_{n}-\bee_{n}|$ we first make use of
\eqref{eqn:wedge-1} and \eqref{eqn:wedge-2}$_1$ to write
\be\label{eqn:diff-1}
\bR\bee_{n}-\bee_{n}
=
\big[\bR\bee_1\wedge\bR\bee_2\wedge\ldots\wedge\bR\bee_{n-1}\big]
-\big[\bee_1\wedge\bee_2\wedge\ldots\wedge\bee_{n-1}\big].
\ee
\n Then, if we subtract and then add terms of the form
\[
\bee_1\wedge\ldots\wedge\bee_{k-1}\wedge\bR\bee_k\wedge
\ldots\wedge\bR\bee_{n-1}
\]
\n to the right-hand side of \eqref{eqn:diff-1}, we find that
\be\label{eqn:diff-2}
\bR\bee_{n}-\bee_{n}
=
\sum_{k=1}^{n-1}\Big[
\bee_1\wedge\ldots\wedge\bee_{k-1}
\wedge(\bR\bee_k-\bee_k)
\wedge\bR\bee_{k+1}\wedge\ldots\wedge\bR\bee_{n-1}\Big].
\ee
\n Taking the norm of \eqref{eqn:diff-2} and making use of the triangle
inequality together with \eqref{eqn:wedge-2}$_2$ and the fact that,
for all $k$, $|\bR\bee_k|=|\bee_k|=1$ yields, with the aid
of \eqref{eqn:tan-small-2},
\[
|\bR\bee_{n}-\bee_{n}|
\le
\sum_{k=1}^{n-1} X|\bR\bee_k-\bee_k|
\le (n-1) G X d r^{\lambda-1},
\]
\n as claimed, which completes the proof.
\end{proof}


\section{Nonlinear Elasticity}\label{sec:NLE}

In the remainder of this manuscript we shall focus on the minimization
problem that arises when one considers the theory of Nonlinear
Elasticity.

\subsection{More Preliminaries}\label{sec:prelim-3}

$\Symn$ will denote the space of \emph{symmetric}
$\bB\in\Mn$, i.e., $\bB=\bB^\rmT$, while $\Psymn$ will denote
those $\bC\in\Symn$ that are \emph{strictly positive definite},
that is, $\ba\cdot\bC\ba>0$ for all nonzero $\ba\in\R^n$.
In the sequel we shall have occasion to consider a function defined
on $\overline{\Omega}\times\cO$, where
$\Omega\subset\R^n$ is a bounded Lipschitz domain
and $\cO\subset\Mn$ is a nonempty, open set.

\begin{definition}\label{def:uniform} Let
$\Phi:\overline{\Omega}\times\cO\to\R$. We say that
$\bF\mapsto\Phi(\bx,\bF)$ is \emph{continuous, almost uniformly in}
$\bx\in\Omega$, \emph{at} $\bF\oo\in\cO$, provided that,
for every $\varepsilon>0$, there
exists a $\delta>0$ such that, for $a.e.~\bx\in\Omega$,
\[
|\Phi(\bx,\bF)-\Phi(\bx,\bF\oo)|<\varepsilon\quad
\text{whenever}\quad
|\bF-\bF\oo|<\delta.
\]
\n More generally, we say that $\bF\mapsto\Phi(\bx,\bF)$
is  $C^2$, \emph{almost uniformly in} $\bx$, \emph{on} $\cO$,
provided $\bF\mapsto\Phi(\bx,\bF)$ and
its first two derivatives are each continuous, almost uniformly
in $\bx\in\Omega$, at every
$\bF\in\cO$.
\end{definition}


\subsection{The Constitutive Relation}\label{sec:CR}

We consider a \emph{body} that for convenience we identify with the
closure of a \emph{bounded Lipschitz domain} $\Omega\subset\R^n$,
$n=2$ or $n=3$, which it
occupies in a fixed reference configuration.  A \emph{deformation}
of $\overline{\Omega}$ is a mapping that lies in the space
\[
\Def:=\{\bu\in W^{1,1}(\Omega;\R^n): \det\grad\bu>0\ a.e.\},
\]
\n where $\det\bF$ denotes the determinant of $\bF\in\Mn$.
We define $\cO\subset \Mn$ by
\[
\cO:=\Mnp=\{\bF\in\Mn: \det\bF>0\}.
\]

We assume that the body is composed of a hyperelastic material with
\emph{stored-energy density}
$W:\overline{\Omega}\times\Mnp\to[0,\infty)$.
$W(\bx,\grad\bu(\bx))$  gives the elastic energy stored
at almost every point $\bx\in\Omega$ of the body when it
undergoes the deformation $\bu\in\Def$.
We assume that the response of the material is
\emph{Invariant under a Change in Observer}
and hence
that\footnote{All of the equations (and inequalities) in this section
are valid only for almost every $\bx\in\Omega$. For clarity of
exposition we have sometimes suppressed this dependence on $\bx$.}
\be\label{eqn:ICO}
 W(\bx,\bQ\bF)=W(\bx,\bF)\quad
\text{for every $\bF\in\Mnp$ and $\bQ\in\SOn$}.
\ee
\n In the sequel we shall have occasion to assume that $W$ also satisfies
(1)--(3) in Hypothesis~\ref{def:W}.   For the moment we suppose that
$\bF\mapsto W(\bx,\bF)$ is $C^2$.


Rather than view the derivatives of $W$ as multilinear forms, as we did
in Section~\ref{sec:I-ELE}, we shall instead follow the usual convention in
Continuum Mechanics (see, e.g., \cite{Ci1988,Gu81});
the (Piola-Kirchhoff) \emph{stress} is
the derivative
\[
\bS(\bx,\bF):= \frac{\partial}{\partial\bF}W(\bx,\bF),
\qquad \bS :\overline{\Omega}\times\Mnp\to\Mn.
\]
\n  The
\emph{Elasticity Tensor} is the second derivative of
$\bF\mapsto W(\bx,\bF)$, that is,
\[
\A(\bx,\bF):=\frac{\partial^2}{\partial\bF^2} W(\bx,\bF),
\qquad
\A:\overline{\Omega}\times\Mnp\to\Lin(\Mn;\Mn),
\]
\n where $\Lin(\mathcal{U};\mathcal{V})$ denotes the
set of linear maps from
the vector space $\mathcal{U}$ to the vector space $\mathcal{V}$.


\begin{remark} In the notation of Section~\ref{sec:I-ELE} and in view of the
symmetry of the second gradient
\[
\begin{gathered}
\bS(\bx,\bF):\bH = \DD W(\bx,\bF)[\bH],\\[2pt]
\bH:\A(\bx,\bF)[\bK]=\bK:\A(\bx,\bF)[\bH] = \DD^2 W(\bx,\bF)[\bH,\bK],
\end{gathered}
\]
\n for all $\bF\in\Mnp$ and all $\bH,\bK\in\Mn$.
\end{remark}

\begin{definition}  The reference configuration is said to be
\emph{stress free} provided that,
\be\label{eqn:SF}
\bS(\bx,\bI) =\mathbf{0}\ \text{ for $a.e.~\bx\in\Omega$.}
\ee
\n If the reference configuration is stress free, then
Elasticity Tensor at the reference
configuration is said to be
\emph{uniformly positive definite}\footnote{One
consequence of \eqref{eqn:ICO} and \eqref{eqn:SF} is that
$\A(\bx,\bI)[\bK]=\mathbf{0}$
for all $\bK\in\Mn$ that satisfy $\bK^\rmT=-\bK$.},
provided that there exists
a constant $c>0$ such that, for every $\bH\in\Mn$
and $a.e.~\bx\in\Omega$,
\[
\bH:\A(\bx,\bI)[\bH]\ge c |\bH+\bH^\rmT|^2.
\]
\end{definition}


  We next note, once again, that
every $\bF\in\Mnp$ has a unique polar
decomposition $\bF=\bR\bU$, where $\bU:=\sqrt{\bF^\rmT\bF}\in \Psymn$
and $\bR:=\bF\bU^{\mi1}\in\SOn$.  Equation
\eqref{eqn:ICO} then implies that $W(\bx,\bF)=W(\bx,\bU)$.
With this in mind we define
$\sigma:\overline{\Omega}\times\Psymn\to\R$  by
\be\label{eqn:def-sigma}
\sigma(\bx,\bC):= W(\bx,\sqrt{\bC}).
\ee
\n Since $\bC\mapsto\sqrt{\bC}$ is $C^\infty$ on  $\Psymn$ our
assumptions
(1)--(3) in Hypothesis~\ref{def:W} yield the same
properties for $\sigma$. In particular, we can differentiate the
identity
\be\label{eqn:W=sigma}
W(\bx,\bF)=W(\bx,\bU)=\sigma(\bx,\bU^2)=\sigma(\bx,\bF^\rmT\bF).
\ee
\n However, we shall need additional smoothness assumptions on $W$ in
order to show that the second variation is uniformly positive near the
set of rotations.   In the sequel we shall therefore sometimes assume
that (see Definition~\ref{def:uniform})
\be\label{eqn:sigma-C2}
\bC\mapsto\sigma(\bx,\bC) \text{ is $C^2$, almost uniformly
in $\bx$, on $\Psymn$,}
\ee
\n and hence, in view of \eqref{eqn:def-sigma}--\eqref{eqn:W=sigma},
that $\bF\mapsto W(\bx,\bF)$ is $C^2$, almost uniformly in
$\bx$, on $\Mnp$.


\begin{remark}\label{rem:small-strain} Note that
\eqref{eqn:W=sigma}
implies that $W$ satisfies \eqref{eqn:ICO}.
\end{remark}


\n The next well-known result shows that our assumptions on
$W$ yield similar properties for $\sigma$.
\begin{lemma}\label{lem:RC}  Let $\sigma$ satisfy
(\ref{eqn:def-sigma})--(\ref{eqn:sigma-C2}).
Then, for all $\bF\in\Mnp$, all $\bH\in\Mn$, and $a.e.~\bx\in\Omega$,
\be\label{eqn:Dsigma-and-D2sigma}
\begin{aligned}
\bS(\bx,\bF)&=2\bF\,\DD\sigma(\bx,\bF^\rmT\bF),\\[3pt]
\bH:\A(\bx,\bF)[\bH]
&=
\big(\bH^\rmT\bF+\bF^\rmT\bH\big)
: \DD^2\sigma(\bx,\bF^\rmT\bF)[\bH^\rmT\bF+\bF^\rmT\bH]\\
&+
2\DD\sigma(\bx,\bF^\rmT\bF):[\bH^\rmT\bH].
\end{aligned}
\ee
\n Moreover, suppose that, for $a.e.~\bx\in\Omega$,
$\bS(\bx,\bI) =\mathbf{0}$ and
\be\label{eqn:positive-definite-2}
\bH:\A(\bx,\bI)[\bH]\ge c |\bH+\bH^\rmT|^2\
\text{ for all $\, \bH\in\Mn$.}
\ee
\n Then, for $a.e.~\bx\in\Omega$,
$\DD\sigma(\bx,\bI)=\mathbf{0}$ and
\be\label{eqn:ET-at-I-PD-alt}
\bB:\DD^2\sigma(\bx,\bI)[\bB]
\ge c |\bB |^2\  \text{ for all $\, \bB\in\Symn$.}
\ee
\n Here $\DD^k\sigma(\bx,\bC)$ denotes the $k$-\emph{th}
derivative of the function $\bC\mapsto\sigma(\bx,\bC)$.
\end{lemma}
\begin{proof}
If we differentiate \eqref{eqn:W=sigma} with respect to $\bF$, we
find that, for all $\bF\in \Mnp$ and $\bH\in\Mn$,
\be\label{eqn:DW-and-Dsigma-1}
\bS(\bx,\bF):\bH
=\DD\sigma(\bx,\bF^\rmT\bF):[\bH^\rmT\bF+\bF^\rmT\bH],
\ee
\n which implies \eqref{eqn:Dsigma-and-D2sigma}$_1$.
If we then differentiate \eqref{eqn:DW-and-Dsigma-1}
with respect to $\bF$ we deduce \eqref{eqn:Dsigma-and-D2sigma}$_2$.
Next, let $\bF=\bI$ in \eqref{eqn:Dsigma-and-D2sigma}$_1$,
to conclude, with the aid of $\bS(\bx,\bI)=\mathbf{0}$, that
$\DD\sigma(\bx,\bI)=\mathbf{0}$.


If we take $\bF=\bI$ in \eqref{eqn:Dsigma-and-D2sigma}$_2$ we
find that
\[
\bH:\A(\bx,\bI)[\bH]=
\big(\bH^\rmT+\bH\big):\DD^2\sigma(\bx,\bI)[\bH^\rmT+\bH],
\]
\n which together with \eqref{eqn:positive-definite-2} yields
\be\label{eqn:positive-definite-3}
\big(\bH^\rmT+\bH\big):\DD^2\sigma(\bx,\bI)[\bH^\rmT+\bH]
\ge c |\bH^\rmT+\bH|^2.
\ee
\n Finally, inequality \eqref{eqn:positive-definite-3} yields
\eqref{eqn:ET-at-I-PD-alt} for all symmetric $\bB$.
\end{proof}


Given a deformation $\bu\in\Def$, the matrix
$\bC_\bu(\bx):=[\bF(\bx)]^\rmT\bF(\bx)$, $\bF:=\grad\bu$, is known as
the \emph{right Cauchy-Green strain tensor.}  It can be used to measure
the change in the length of a curve in the reference configuration after
it is deformed by $\bu$.  The matrix
\be\label{eqn:strain}
\bE_\bu(\bx):=\tfrac12(\bC_\bu(\bx)-\bI)
=
\tfrac12\big([\bF(\bx)]^\rmT\bF(\bx)-\bI\big)
\ee
\n is sometimes referred to as the (nonlinear)
\emph{strain.}\footnote{See, e.g., \cite[Section~2.2.7]{Og84} for
a discussion
of various measures of strain.}  The linearization of $\bE$ at
$\bF=\bI$ yields the strain tensor used in the classical theory
of Linear Elasticity.  The advantage of using $\bE$, rather
than $\bC$, is that $\bE=\mathbf{0}$ corresponds to an undeformed body.
We next note that
\emph{a deformation has uniformly small strains if
and only if it is uniformly close to the set of rotations.}

\begin{lemma}\label{lem:strain=dist-to-rotations} Let $\bF\in\Mnp$.
Then
\be\label{eqn:strain=dist-to-rotations}
\big[\dist\!\big(\bF,\SOn\big)\big]^2 \le
2\sqrt{n}\,|\bE|
\le
\sqrt{n}\dist\!\big(\bF,\SOn\big)
\big[\dist\!\big(\bF,\SOn\big) +2\sqrt{n}\,\big].
\ee
\end{lemma}


\begin{proof}  Define $\bA\in \Mn$ by
$\bA:=\diag\{|a_1|,|a_2|,\ldots,|a_n|\}$, where
$a_k\in\R$.  Then, by the Cauchy-Schwarz inequality,
\be\label{eqn:holder-counting}
\bigg(\sum_{k=1}^n |a_k| \bigg)^{\!2}
=|\bA:\bI|^2\le |\bA|^2 |\bI|^2
=  n\sum_{k=1}^n |a_k|^2.
\ee
\n Next, by the spectral theorem, $\bU=\sqrt{\bC}$
has eigenvalues $0<\lambda_1\le\lambda_2\le\cdots\le\lambda_n$.
Since $|\lambda_k-1|^2\le |\lambda_k^2-1|$ the choice $a_k=\lambda_k^2-1$
in \eqref{eqn:holder-counting} yields, with the aid of
\eqref{eqn:strain},
\[ 
|\bU-\bI|^4=\bigg(\sum_{k=1}^n |\lambda_k-1|^2 \bigg)^{\!2}
\le  n\sum_{k=1}^n |\lambda_k^2-1|^2 = 4n|\bE|^2,
\] 
\n which together with Lemma~\ref{lem:alt-dist} establishes the
first inequality in \eqref{eqn:strain=dist-to-rotations}.

The identity $\bC=\bU^2$
together with \eqref{eqn:strain}, Lemma~\ref{lem:alt-dist}, and
the triangle inequality gives us
\[
\begin{aligned}
2|\bE|
&=
|(\bU-\bI)(\bU+\bI)|
\le
\dist\!\big(\bF,\SOn\big)\big(|\bU|+\sqrt{n}\,\big),\\
|\bU| &=|\bU-\bI+\bI|\le \dist\!\big(\bF,\SOn\big) +\sqrt{n},
\end{aligned}
\]
\n which together yield the second inequality in
\eqref{eqn:strain=dist-to-rotations}.
\end{proof}

\begin{remark}\label{rem:d<2E}  We note for future reference that
$|\lambda_k-1|\le|\lambda_k-1||\lambda_k+1|=|\lambda_k^2-1|$ and hence,
in view of  Lemma~\ref{lem:alt-dist} and \eqref{eqn:strain},
\be\label{eqn:d<2E}
\big[\dist\!\big(\bF,\SOn\big)\big]^2=|\bU-\bI|^2=\sum_{k=1}^n
|\lambda_k-1|^2 \le \sum_{k=1}^n |\lambda_k^2-1|^2 = 4|\bE|^2.
\ee
\n Although \eqref{eqn:d<2E} does not scale properly for large
strains, its use will simplify the small strain computation in one
of our proofs.
\end{remark}


\subsection{Equilibrium Solutions and Energy Minimizers in Nonlinear
Elasticity}\label{sec:ES-EM-NLE}

 We assume the body is subject to
dead loads.  As in Section~\ref{sec:I-ELE} we shall let
\[
\partial\Omega = \overline{\sD} \cup \overline{\sS}\quad
\text{with $\sD$ and $\sS$ relatively open and }
\sD\cap\sS=\varnothing.
\]
\n In addition, we shall suppose that $\sD\ne\varnothing$.
We assume that a Lipschitz-continuous function
$\bd:\sD\to\R^n$
is prescribed;  $\bd$ will give the deformation of $\sD$.
If $\sS \ne\varnothing$ we assume that a function
$\bs\in L^2(\sS;\R^n)$
is prescribed; for $\sH^{n-1}$-$a.e.~\bx\in\sS$,
$\bs(\bx)$ will give the
surface force
exerted on the body, at the point $\bx$, by its environment.
Finally, we suppose that a
function $\bb\in L^2(\Omega;\R^n)$ is prescribed;
for $a.e.~\bx\in\Omega$, $\bb(\bx)$ will give the
body force
exerted on the body, at the point $\bx$, by its environment.
  The set of
\emph{Admissible Deformations} will be denoted by
\[
\AD:=\{\bu\in\Def \cap\, W^{1,\infty}(\Omega;\R^n):
\bu=\bd \text{ on $\sD$}\}.
\]
\n The  \emph{total energy} of an admissible deformation
$\bu\in\AD$ is defined to be
\be\label{eqn:TE}
\E(\bu):=\int_\Omega \big[W\big(\bx,\grad\bu(\bx)\big)
-\bb(\bx)\cdot\bu(\bx)\big]\dd\bx
-\int_{\sS} \bs(\bx)\cdot\bu(\bx)\,\dd\sH^{n-1}_\bx.
\ee
\n We shall assume that we are given a deformation,
$\bu\e\in\AD$, that is a weak
solution of the \emph{Equilibrium Equations} corresponding
to \eqref{eqn:TE}, i.e.,
\be\label{eqn:EE}
0=\int_\Omega
\big[\bS\big(\bx,\grad\bu\e(\bx)\big):\grad\bw
-\bb(\bx)\cdot\bw(\bx)\big]\dd\bx
-\int_\sS \bs(\bx)\cdot\bw(\bx)\,\dd \sH^{n-1}_\bx
\ee
\n for all \emph{variations} $\bw\in\Var$, where
\[
\Var:=\{\bw\in W^{1,2}(\Omega;\R^n):
\bw=\mathbf{0} \text{ on $\sD$}\}.
\]


 If $\sD=\partial\Om$ we shall call $\bu\e$ a weak solution of the
\emph{(pure) displacement problem}.  Otherwise, we shall refer to
such a  $\bu\e$ as a weak solution of the
\emph{(genuine) mixed problem}.
If, in addition, $W\in C^2(\Omega\times\Mnp)$ and
$\bu\e\in C^2(\Omega;\R^n)\cap C^1(\overline{\Omega};\R^n)$, then
$\bu\e$ will be a \emph{classical solution of the equations of
equilibrium}, i.e., 
\[
\begin{gathered}
\Div \bS(\grad\bu\e) +\bb= \mathbf{0}\
\text{ in $\Omega$,}\\
 \bS(\grad\bu\e)\bn = \bs\
\text{ on $\sS$,}\qquad
\bu\e=\bd\
\text{ on $\sD$,}
\end{gathered}
\]
\n where $\bn(\bx)$ is the outward unit normal to $\Omega$ at
$\sH^{n-1}$-$a.e.~\bx\in\sS$ and $\Div \bS\in\R^n$ is given by
$(\Div \bS)_i=\sum_j \frac{\partial}{\partial \bx_j}\bS_{ij}$.
We are interested in conditions under which a weak solution of the
equilibrium equations, $\bu\e\in\AD$, is a local minimizer of the
total energy $\E$.  We are also interested in conditions
under which  $\bu\e$ is the unique weak solution of the
equilibrium equations that lies in a neighborhood of $\bu\e$.


\subsection{Uniqueness in $\BMO\cap\, L^1$
Neighborhoods in Elasticity}

We next make note of a direct implication of
Theorem~\ref{thm:SVP=LMBMO} for Elasticity.


\begin{theorem} \label{thm:NT-3.0} Let $W$ satisfy (1)--(3) of
Hypothesis~\ref{def:W}. Suppose that $\bu\e\in\AD$ is a weak solution
of the pure-displacement or mixed problem that satisfies,
for some $\varepsilon>0$ and $k>0$,
\[
\text{$\det\grad\bu\e> \varepsilon\ \ a.e.$}, \qquad
\int_\Omega \grad\bw:\A(\grad\bu\e)[\grad\bw]\,\dd\bx
\ge
4k\int_\Omega |\grad\bw|^2\dd\bx,
\]
\n for all $\bw\in\Var$.  Let $\tau\in\R$ satisfy
$\tau>||\grad\bu\e||_{\infty,\Om}$ and $\tau^{\mi1}<\varepsilon$.
Then there exists a
$\delta=\delta(\tau)>0$ such that any $\bv\in\AD$ that
satisfies  $\det\grad\bv>\tau^{\mi1}~a.e.$,
\be\label{eqn:small+e-t}
||\grad\bv||_{\infty,\Om} <\tau,
\qquad
\tsn{\grad\bv-\grad\bu\e}_{\BMO(\Omega)}<\delta,
\qquad
\Big|\dashint_\Omega (\grad\bv-\grad\bu\e)\,\dd\bx\Big|<\delta,
\ee
\n will also satisfy
\[
\E(\bv)\ge \E(\bu\e)
+k\int_\Omega|\grad\bv-\grad\bu\e|^2\dd\bx.
\]
\n  In particular, $\bv\not\equiv\bu\e$ will have strictly greater
energy than $\bu\e$.  Moreover,  $\bv$ cannot be a weak solution
of the equations of equilibrium, \eqref{eqn:EE}.
\end{theorem}


A physical interpretation of hypothesis \eqref{eqn:small+e-t}$_2$
is of interest.
In the remainder of the paper we will show that, in certain situations,
sufficiently small strains or small strain differences imply that
\eqref{eqn:small+e-t}$_2$ is satisfied.


\section{Deformations with Small Strain}\label{sec:Def-SS}

In this section we focus on deformations $\bu$ whose
nonlinear strains $\bE_\bu$
are sufficiently small.  We show, in particular, that uniformly small
strains implies that the deformation gradient is small in $\BMO$.


\subsection{The Positivity of the Second Variation
for Deformations with Small Strain}\label{sec:ES-EM}

We now consider the sign of the second variation for deformations
that have sufficiently small strains.  The next result
shows that a stress-free reference configuration together with the
uniform positivity of the Elasticity Tensor at this reference
configuration yields the uniform positivity of the second variation of
the total energy at any admissible deformation, $\bu\in\AD$, that
either is $C^1$ and has sufficiently small strains, or is sufficiently
close to a single rotation.


\begin{proposition}\label{prop:positive-SV}  Let
$\bF\mapsto W(\bx,\bF)$ be $C^2$, almost uniformly in $\bx$,
on $\Mnp$ and satisfy \eqref{eqn:ICO}.
Suppose that, for $a.e.~\bx\in\Omega$, $\bS(\bx,\bI)=\mathbf{0}$ and
\be\label{eqn:ET-at-I-PD-third}
\bH:\A(\bx,\bI)[\bH]\ge c |\bH+\bH^\rmT|^2
\ee
\n for some constant $c>0$ and every $\bH\in\Mn$.
Then there exists a $\delta\oo\in(0,1)$
such that any admissible deformation $\bu\in\AD$
that satisfies both
\be\label{eqn:rot-dist-small-L-infty}
\bu\in C^1(\overline{\Omega};\R^n) \quad \text{and} \quad
||(\grad\bu)^\rmT\grad\bu-\bI||_{\infty,\Omega} < \delta\oo
\ee
\n or, merely,
\be\label{eqn:I-dist-small-L-infty}
||\grad\bu-\bQ||_{\infty,\Omega} < \delta\oo
\ee
\n for some $\bQ\in\SO(n)$, will also satisfy
\be\label{eqn:pos-SV-ELAS}
\int_\Omega \grad\bw(\bx)
:\A\big(\bx,\grad\bu(\bx)\big)\big[\grad\bw(\bx)\big]\,\dd\bx
\ge
4k\int_\Omega |\grad\bw(\bx)|^2\dd\bx,
\ee
\n for some $k>0$ and all $\bw\in\Var$.
\end{proposition}


\begin{remark}  (1).~Note that Lemma~\ref{lem:strain=dist-to-rotations}
and \eqref{eqn:rot-dist-small-L-infty}$_2$ imply that the
distance from $\grad\bu$  to the set of rotations is small.
 (2).~The additional smoothness of $\bu$,
\eqref{eqn:rot-dist-small-L-infty}$_1$, is necessitated
by our use of a version of Korn's inequality with nonconstant
coefficients.  See Appendix~\ref{sec:Korn}.
\end{remark}


\begin{proof}[Proof of Proposition~\ref{prop:positive-SV}]  We first note
that the result is well-known when $\grad\bu$ satisfies
\eqref{eqn:I-dist-small-L-infty} (see, e.g., \cite[Theorem~5]{GS79}).
We shall therefore assume that $\bu\in\AD$ satisfies
\eqref{eqn:rot-dist-small-L-infty} for
some $\delta\oo\in(0,1)$ to be determined.
Suppose that $\varepsilon>0$ is an additional small parameter
to be determined.
Then, by hypothesis and Lemma~\ref{lem:RC},
$\bS(\bx,\bI)=\DD\sigma(\bx,\bI)=\mathbf{0}$.  The continuity of
$\DD\sigma$ (almost uniformly in $\bx$) then yields an
$\eta>0$ such that,
for $a.e.~\bx\in\Omega$,
\be\label{eqn:Dsigma-small}
|\DD\sigma(\bx,\bF^\rmT\bF)|< \varepsilon\
\text{ whenever }  \ |\bF^\rmT\bF-\bI|<\eta.
\ee
\n Thus, in view of \eqref{eqn:rot-dist-small-L-infty}$_2$, if we choose
$\delta\oo < \eta$, it follows that, for $a.e.~\bx\in\Omega$,
\be\label{eqn:stress-small-one}
2|\DD\sigma(\bx,\bF^\rmT\bF):(\bH^\rmT\bH)|<2\varepsilon|\bH|^2.
\ee


We next consider
\[
\big(\bH^\rmT\bF+\bF^\rmT\bH\big):
\DD^2\sigma(\bx,\bF^\rmT\bF)[\bH^\rmT\bF+\bF^\rmT\bH].
\]
\n   Define
$\bB:=\bH^\rmT\bF+\bF^\rmT\bH\in\Symn$
and rewrite this quadratic form (in $\bB$) as
\be\label{eqn:split-one}
\begin{aligned}
\bB:\DD^2\sigma(\bx,\bF^\rmT\bF)[\bB] = \bB:\DD^2\sigma(\bx,\bI)[\bB]
+\bB:\big(\DD^2\sigma(\bx,\bF^\rmT\bF)-\DD^2\sigma(\bx,\bI)\big)[\bB].
\end{aligned}
\ee
\n Then, given $\varepsilon>0$, the continuity of $\DD^2\sigma$
(almost uniformly in $\bx$)
yields a $\beta>0$ such that, for $a.e.~\bx\in\Omega$,
\be\label{eqn:D2-sigma-cont-one}
\big|\DD^2\sigma(\bx,\bF^\rmT\bF)-\DD^2\sigma(\bx,\bI)\big|
< \varepsilon\ \text{ whenever }  \ |\bF^\rmT\bF-\bI|<\beta.
\ee
\n In view of \eqref{eqn:rot-dist-small-L-infty}$_2$,
a choice of $\delta\oo < \beta$ yields
\be\label{eqn:diff-small-one}
\big|\bB:\big(\DD^2\sigma(\bx,\bF^\rmT\bF)
-
\DD^2\sigma(\bx,\bI)\big)\big[\bB\big]\big|
\le
\varepsilon|\bB|^2,
\ee
\n for $a.e.~\bx\in\Omega$.
Lastly, in view of \eqref{eqn:ET-at-I-PD-third} and
Lemma~\ref{lem:RC}, the remaining term in
\eqref{eqn:split-one} satisfies
\be\label{eqn:pos-def-one}
\bB:\DD^2\sigma(\bx,\bI)[\bB] \ge c|\bB|^2.
\ee


\n If we let  $\bF=\grad\bu(\bx)$ and $\bH=\grad\bw(\bx)$ in
\eqref{eqn:Dsigma-and-D2sigma}$_2$, integrate over $\Omega$,
and make use of \eqref{eqn:stress-small-one}, \eqref{eqn:split-one},
\eqref{eqn:diff-small-one}, and \eqref{eqn:pos-def-one},
we conclude that
\be\label{eqn:two-bound-A-below-one}
\begin{aligned}
\int_\Omega \grad\bw : \A(\grad\bu)[\grad\bw]\,\dd\bx
\ge
(c-\varepsilon)
&\int_\Omega \big|(\grad\bw)^\rmT\grad\bu
+(\grad\bu)^\rmT\grad\bw\big|^2 \,\dd\bx\\
-2\varepsilon&\int_\Omega |\grad\bw|^2\,\dd\bx.
\end{aligned}
\ee


We now assume that $\bu\in C^1(\overline{\Omega};\R^n)$.
A generalized Korn's inequality, Proposition~\ref{prop:Korn-1},
then yields the existence of
a constant $K>0$ such that
\[
\int_\Omega
\big|(\grad\bw)^\rmT\grad\bu+(\grad\bu)^\rmT\grad\bw\big|^2\,\dd\bx
\ge
K \int_\Omega|\grad\bw|^2\,\dd\bx,
\]
\n which together with
\eqref{eqn:two-bound-A-below-one} gives us
\be\label{eqn:two-bound-A-below-two}
\int_\Omega \grad\bw : \A(\grad\bu)[\grad\bw]\,\dd\bx
\ge
\big[K(c-\varepsilon)-2\varepsilon\big]\int_\Omega |\grad\bw|^2\,\dd\bx.
\ee
\n Finally, we return to
$\varepsilon$ and $\delta\oo$.
Choose $\varepsilon>0$ that satisfies
$\varepsilon<\min\{c,Kc/(K+2)\}$ so that
\eqref{eqn:two-bound-A-below-two}
will yield \eqref{eqn:pos-SV-ELAS}.  Then choose $\delta\oo>0$
so that $\delta\oo<\min\{\eta,\beta,1\}$,
where $\eta$ and $\beta$ are determined by $\varepsilon$ in
\eqref{eqn:Dsigma-small} and \eqref{eqn:D2-sigma-cont-one},
respectively.  That concludes the proof.
\end{proof}


\subsection{Uniqueness of Equilibrium that have
Sufficiently Small Strains}\label{sec:unique-elasticity-2}

We are now ready to apply the results obtained for general integrands
in the Calculus of Variations to elastic deformations with small
strains.

\begin{theorem}\label{thm:unique}  Let $\bF\mapsto W(\bx,\bF)$
be $C^2$, almost uniformly in $\bx$, on $\Mnp$.  Suppose that
$W$ satisfies
\eqref{eqn:ICO} and (1)--(3) of Hypothesis~\ref{def:W}.
Assume, in addition, that, for $a.e.~\bx\in\Omega$,
$\bS(\bx,\bI)=\mathbf{0}$ and
\[
\bH:\A(\bx,\bI)[\bH]\ge c |\bH+\bH^\rmT|^2
\]
\n for some constant $c>0$ and every $\bH\in\Mn$.
Then there exists
a $\delta\in(0,1)$ such that any solution, $\bu\e\in\AD$,
of the equilibrium equations \eqref{eqn:EE}, for either the
pure-displacement problem or the mixed problem,
that satisfies both
\be\label{eqn:dist-to-rot-small-end-1}
\bu\e\in C^1(\overline{\Omega};\R^n) \quad \text{ and } \quad
||(\grad\bu\e)^\rmT\grad\bu\e-\bI||_{\infty,\Omega} < \delta
\ee
\n or, merely,
\be\label{eqn:dist-to-rot-small-end-1-again}
||\grad\bu\e-\bQ||_{\infty,\Omega} < \delta
\ee
\n for some $\bQ\in\SOn$,
\n is the unique minimizer of the energy among
$\bv\in\AD$ that satisfy
\be\label{eqn:dist-to-rot-small-end-2}
||(\grad\bv)^\rmT\grad\bv-\bI||_{\infty,\Omega} < \delta.
\ee
\n Moreover, there are no other equilibrium solutions,
$\buh\e\in\AD$, that satisfy
\eqref{eqn:dist-to-rot-small-end-2} with $\bv=\buh\e$.
\end{theorem}


Theorem~\ref{thm:unique} establishes that there is at most
one solution with (sufficiently) small strains for both the
pure-displacement and the mixed problem in Nonlinear Elasticity.
For the pure-displacement problem, essentially the same result
(with a similar proof) was
first established by John~\cite{Jo72}.  A more recent
elementary proof, under different hypotheses, can be
found in \cite{SS18}.

\begin{remark} Theorem~\ref{thm:unique} does not yield the
\emph{existence} of any solutions of the equilibrium
equations that satisfy \eqref{eqn:dist-to-rot-small-end-1}.
However, suppose that the stored-energy density, the boundary, and
the data: $(\bd,\bs,\bb)$
are sufficiently smooth and either $\mathcal{D}=\partial\Omega$
(the displacement problem) or both $\partial\mathcal{S}=\varnothing$
and  $\partial\mathcal{D}=\varnothing$, e.g., a thick spherical shell
with $\mathcal{S}$ and $\mathcal{D}$ the inner and outer boundaries.
Then results
of Valent~\cite{Va88}, which make use of estimates for systems of linear
elliptic equations and the implicit function theorem,
yield the existence of a solution that
satisfies \eqref{eqn:dist-to-rot-small-end-1} whenever $\bs$ and
$\bb$ are sufficiently small and $\bd$ is sufficiently close
to the identity.
\end{remark}

\begin{remark} In Theorem~\ref{thm:unique} it is irrelevant whether or
not the equilibrium solution is injective.
This may engender curious consequences.  For example,
suppose that one can show that a non-injective equilibrium
solution with (sufficiently) small strains exists.  Then
Theorem~\ref{thm:unique} implies, in particular, that there are
\emph{no
injective equilibrium solutions with small strains.}
\end{remark}


\begin{proof}[Proof of Theorem~\ref{thm:unique}]  We shall assume that
$\bu\e$ satisfies \eqref{eqn:dist-to-rot-small-end-1}.  The proof
when $\bu\e$ satisfies \eqref{eqn:dist-to-rot-small-end-1-again}
is similar.   Let
$\bu\e\in\AD$ be a solution of \eqref{eqn:EE}
that satisfies
\eqref{eqn:dist-to-rot-small-end-1} for some $\delta\in(0,1)$ to
be determined.
Then, in view of Proposition~\ref{prop:positive-SV}, there exists
a $\delta\oo\in(0,1)$ and a $k>0$ such that, for all $\bw\in\Var$,
\be\label{eqn:pos-SV-ELAS-two}
\int_\Omega \grad\bw:\A(\grad\bu\e)[\grad\bw]\,\dd\bx
\ge
4k\int_\Omega |\grad\bw|^2\dd\bx,
\ee
\n provided $\delta<\delta\oo$.
Now, let $\bv\in\AD$ satisfy \eqref{eqn:dist-to-rot-small-end-2}
for some $\delta\in(0,\delta\oo)$ to be determined.


Next, fix $p>n$.  Then Proposition~\ref{prop:BC+SR=close-in-L1} together
with \eqref{eqn:strain}, \eqref{eqn:d<2E},
\eqref{eqn:dist-to-rot-small-end-1}$_2$, and
\eqref{eqn:dist-to-rot-small-end-2} yield a constant $A^*>0$ such that
\be\label{eqn:small-in-L1}
||\grad\bu\e-\grad\bv||_{1,\Omega}< 2A^*|\Omega|^{1/p}\delta.
\ee
\n Also, in view of Proposition~\ref{prop:GR} (Geometric Rigidity),
there exists a constant $M>0$ such that
\[
\tsn{\grad\bu\e}_{\BMO(\Omega)} < M\delta,
\qquad
\tsn{\grad\bv}_{\BMO(\Omega)} < M\delta,
\]
\n and hence, by the triangle inequality,
\be\label{eqn:small-in-BMO}
\tsn{\grad\bu\e-\grad\bv}_{\BMO(\Omega)} < 2M\delta.
\ee
\n Finally, if we define
\[
\sB:=\{\bF\in\Mn: \dist(\bF,\SOn) < \delta < 1\}\subset\Mnp,
\]
\n we find that, for almost every $\bx\in\Omega$,
\be\label{eqn:in-B}
\grad\bu\e(\bx)\in\sB, \qquad \grad\bv(\bx)\in\sB.
\ee

We now take note of \eqref{eqn:pos-SV-ELAS-two}, \eqref{eqn:small-in-L1},
\eqref{eqn:small-in-BMO}, and \eqref{eqn:in-B}
and choose
$\delta\in(0,\delta\oo)$ sufficiently small so that
$\bu\e$ and $\bv$ satisfy the hypotheses of Theorem~\ref{thm:SVP=LMBMO}.
We then find that  $\bu\e$ and $\bv$ satisfy the conclusions of that
theorem, i.e.,
\[
\E(\bv)\ge \E(\bu\e)
+k\int_\Omega|\grad\bv-\grad\bu\e|^2\dd\bx;
\]
\n  $\bv\not\equiv\bu\e$ has strictly greater energy than $\bu\e$;
and $\bv\not\equiv\bu\e$ cannot be an equilibrium solution.
\end{proof}

\section{Uniqueness of Equilibrium with Sufficiently Small Strain
Differences; Change of Reference Configuration}\label{sec:CRC}



In this section we extend the uniqueness results obtained in
Section~\ref{sec:unique-elasticity-2}.
In particular, we show that the positivity of the second variation at a
weak solution of the equilibrium equations, $\bu\e$, that is a
diffeomorphism, implies that $\bu\e$ is a strict minimizer of the
energy among those admissible deformations $\bv$ whose right
Cauchy-Green strain tensor $\bC_\bv:=(\grad\bv)^\rmT\grad\bv$
is uniformly and sufficiently close to
$\bC\e:=(\grad\bu\e)^\rmT\grad\bu\e$.
We also show that such a $\bv$ cannot be a weak solution
of the equilibrium equations.  We begin with some additional notations.


\emph{Recall that we consider a body that we identify with the
closure of a
bounded, Lipschitz domain} $\Om\subset \R^n$, $n=2$ or $n=3$,
that it occupies in a fixed reference configuration.
We let $C^0(\overline{\Omega};\R^n)$ denote those maps
$\bu:\overline{\Omega}\to\R^n$ that are bounded and uniformly continuous
on the closure of $\Omega$.  We shall write
$\bu\in C^1(\overline{\Omega};\R^n)$ provided that both $\bu$ and its
classical gradient $\grad\bu$ are bounded and uniformly continuous
on the closure of $\Omega$.  Note that, for each $\bx\in \Omega$,
$\grad\bu(\bx)\in\Mn$ with components
$[\grad\bu]_{ij}= \partial u_i/\partial x_j$.
As in Section~\ref{sec:ES-EM-NLE},
we shall let
\[
\partial \Om = \overline{\sD} \cup \overline{\sS}  \quad
\text{with $\sD$ and $\sS$ relatively open, }\
\sD\cap\sS =\varnothing,
\]
\n and $\sD\ne \varnothing$.  In addition, we suppose that functions
$\bd\in  C^1(\overline{\sD};\R^n)$, $\bb\in L^2(\Om;\R^n)$, and,
if $\sS \ne\varnothing$, $\bs\in L^2(\sS;\R^n)$
are prescribed.  We assume that $\bd$ is one-to-one.


We next define what we mean by a diffeomorphism and we also
recall our definition of
admissible deformations and variations from Section~5.

\begin{definition}  Let $\bu:\overline{\Omega} \to \R^n$ be an injective
mapping with inverse
$\bu^{\mi1}:\bu(\overline{\Omega})\to \overline{\Omega}$.  We
call $\bu$ an (orientation preserving) \emph{diffeomorphism} provided
that  \vspace{-.2cm}
\begin{enumerate}
\item $\bu\in C^1(\overline{\Omega};\R^n)$;
\item $\bu^{\mi1}\in C^1(\bu(\overline{\Omega});\R^n)$; and
\item $\det\grad\bu > 0$ on the compact set $\overline{\Omega}$.
\vspace{-.2cm}
\end{enumerate}
\n Next, recall that
\be\label{eqn:var-and-AD}
\begin{gathered}
\AD:=\{\bu\in W^{1,\infty}(\Om;\R^n):
\det\grad\bu>0~a.e.,
\ \bu=\bd \text{ on } \sD\},\\[2pt]
\Var:=\{\bw\in W^{1,2}(\Om;\R^n):
\bw=\mathbf{0} \text{ on $\sD$}\}.
\end{gathered}
\ee
\end{definition}


The main result of this section is the following theorem.

\begin{theorem}\label{thm:NT-3} Let $W$ satisfy (1)--(3)
of Hypothesis~\ref{def:W}.  Suppose that  \vspace{-.2cm}
\begin{enumerate}
\item[(A)] $\bu\e\in \AD$ is a diffeomorphism;
\item[(B)] $\bu\e$ is a weak solution of the equilibrium
equations; and
\item[(C)] $\bu\e$  satisfies
\[
\int_\Om \grad\bw(\bx)\Td
\A\big(\bx,\grad\bu\e(\bx)\big)\big[\grad\bw(\bx)\big]\,\dd\bx
\ge
4k\int_\Om |\grad\bw(\bx)|^2\dd\bx,
\]
\n for some $k>0$ and all $\bw\in\Var$.
\end{enumerate}
\n Then there exists an $\varepsilon>0$ such that any
$\bv\in \AD$ that satisfies 
\be\label{eqn:dist-small-ref}
0<\big\|(\grad\bv)^\rmT\grad\bv
-(\grad\bu\e)^\rmT\grad\bu\e\big\|_{\infty,\Om}
< \varepsilon
\ee
\n has strictly greater energy than $\bu\e$.  Moreover, there are no
other weak solutions of the equilibrium equations,
$\bv\e\in \AD$, that
satisfy \eqref{eqn:dist-small-ref} with $\bv=\bv\e$.
\end{theorem}


\begin{remark}\label{rem:CM-15}  Our proof of Theorem~\ref{thm:NT-3}
requires that we show that all of the hypotheses of
Theorem~\ref{thm:SVP=LMBMO} are satisfied.  A direct application of
Theorem~\ref{thm:SVP=LMBMO} would necessitate us to make use of
\eqref{eqn:dist-small-ref} to demonstrate that $\bu\e$ and $\bv$
satisfy \eqref{eqn:small+B}$_{2,3}$.  In this regard, Ciarlet \&
Mardare~\cite{CM15} have obtained extensions of the Geometric-Rigidity
results of \cite{FJM02} and \cite{CS06} (Proposition~\ref{prop:GR} in
this manuscript) that include a second mapping
$\bu\e\in C^1(\Omc;\R^n)$ with $\det\grad\bu\e>0$ on $\Omc$,
but which need not be injective. Their results imply that there
exists a constant $K=K(p,\bu\e,\Om)$ such that any
$\bv\in W^{1,p}(\Om;\R^n)$, $2\le p<\infty$, that satisfies
$\det\grad\bv>0~a.e.$ and $\bv=\bu\e$ on $\sD$ will also satisfy
\be\label{eqn:GR-CM}
\|\grad\bv-\grad\bu\e\|^2_{p,\Om} \le
K\big\|(\grad\bv)^\rmT\grad\bv
-(\grad\bu\e)^\rmT\grad\bu\e\big\|_{p/2,\Om}.
\ee
\n This result together with \eqref{eqn:dist-small-ref} yields the
integral estimate \eqref{eqn:small+B}$_3$.  However,
Theorem~\ref{thm:SVP=LMBMO} also requires the $\BMO$-estimate
\eqref{eqn:small+B}$_2$. Unfortunately, a $\BMO$-estimate such as
\eqref{eqn:small-dist-&-small_BMO} does not follow from
\eqref{eqn:GR-CM} due to the dependence of the constant $K$ upon the
mapping $\bu\e$.  To obtain \eqref{eqn:small-dist-&-small_BMO}
from \eqref{eqn:GR} one must make use of the fact that the constant
$C$ in \eqref{eqn:GR} is the same for all cubes contained in the region.
\end{remark}

\begin{remark}\label{rem:final-last} At the end of the introduction we
noted that it would be of interest to prove some of our results, e.g.,
Theorem~\ref{thm:NT-3},
without the assumption that $\bu\e$ is one-to-one on $\overline{\Omega}$.
This is of particular interest when the restriction of $\bu\e$ to
$\sS$ is not one-to-one and the deformed body then exhibits
self-contact (see, e.g., Ciarlet~\cite[Section~5.6]{Ci1988}). The main
difficulty is that the boundary of $\bu\e(\Om)$ may then fail
to be Lipschitz since the deformed region may be on both ``sides''
of its boundary. Here one might want to attempt to follow the approach
in \cite{CM15} that partitions $\Om$ into
subdomains upon which $\bu\e$ is injective.
We also note that much of our
proof is valid if $\bu\e$ is bi-Lipschitz, rather than a
diffeomorphism.  However, once again, $\bu\e(\Om)$ may then fail to
be Lipschitz.
See the counterexample in \cite[Section~1.2]{Gr85}.
\end{remark}


Our proof of Theorem~\ref{thm:NT-3} involves a change in reference
configuration.\footnote{A change in reference configuration
is a standard procedure in Continuum Mechanics. See, e.g.,
Ciarlet~\cite[Chapter~1]{Ci1988}.} This change of variables will show
that our assumption that two strain tensors are close to each
other yields
a new deformation whose gradient is close to the set of rotations.
We shall then make use of Geometric Rigidity and
Theorem~\ref{thm:SVP=LMBMO}.  We postpone the proof of
Theorem~\ref{thm:NT-3} to the end of this section.

\subsection{Change of Reference Configuration}\label{subsec:CRC}

In this subsection we present the required change of variables
that makes the deformed configuration into a new reference
configuration.  Those readers already familiar with this
procedure may prefer to skip to Section~\ref{sec:Proof}.

\subsubsection{The Body and its Deformed Image}\label{sec:body}
We first recall some properties of domains and their image
under injective mappings.
Let $U\subset \R^n$, $n\ge2$,
be a bounded domain.  Suppose that $\bu\in C^0(U;\R^n)$
is injective.  Then standard results in topology and degree theory
(see, e.g., \cite[Theorem~3.30]{FG95}) imply that $\bu(U)$ is also
a bounded domain.  Since $\Omega\subset \R^n$, $n\ge2$, denotes
a bounded Lipschitz domain, when $\bu\in C^0(\overline{\Omega};\R^n)$
is injective it then follows that
$\bu(\partial\Omega)=\partial\bu(\Omega)$.  The next result is
well known.  We sketch
the proof for the interested reader.


\begin{proposition}\label{prop:compositions}  Suppose that
$\bu\in C^1(\overline{\Omega};\R^n)$ is a diffeomorphism.
Then $\bu(\Omega)$
is a bounded Lipschitz domain; $\bu$ and $\bu^{\mi1}$ satisfy
\be\label{eqn:der-inv}
\big[\grad_\bx \bu(\bx)\big]^{\mi1}
=
\grad_\by \bu^{\mi1}(\by)\
\text{ with $\by=\bu(\bx)$.}
\ee
\n Moreover, if
$\bzh\in W^{1,p}(\bu(\Omega);\R^n)$ and
$\bw\in W^{1,p}(\Omega;\R^n)$, $p\in[1,\infty]$, then
$\bzh\circ\bu\in W^{1,p}(\Omega;\R^n)$,
$\bw\circ\bu^{\mi1}\in W^{1,p}(\bu(\Omega);\R^n)$, and
\be
\begin{aligned}\label{eqn:chain-rule-1}
\grad_\bx (\bzh\circ\bu)(\bx)
&=
\big[\grad_\by \bzh\big(\bu(\bx)\big)\big]
\grad\bu(\bx)\
\text{ for $a.e.~\bx\in\Omega$},\\[2pt]
\grad_\by (\bw\circ\bu^{\mi1})(\by)
&=
[\grad_\bx \bw\big(\bu^{\mi1}(\by)\big)\big]
\grad \bu^{\mi1}(\by)
\
\text{ for $a.e.~\by\in\bu(\Omega)$}.
\end{aligned}
\ee
\end{proposition}


\begin{remark}\label{rem:a.e.}  We note that the change of variables
formula also shows that
diffeomorphisms map sets of measure zero to sets of
measure zero, e.g., if $\det\grad\bv(\bx)>0$ for $a.e.~\bx\in\Om$
then $\det\grad(\bv\circ\bu^{\mi1})(\by)>0$ for $a.e.~\by\in\bu(\Om)$.
\end{remark}

\begin{proof}[Sketch of the proof of
Proposition~\ref{prop:compositions}]
The set $\bu(\Omc)$ is compact and hence bounded.
Equation~\eqref{eqn:der-inv} follows from the chain rule for
diffeomorphisms.  We next show
that $\bu(\Omega)$ is a Lipschitz domain.  We note that a
result of Whitney~\cite{Wh34} implies that the Whitney extension theorem
(see, e.g., \cite[Section~6.5]{EG92})
applies to Lipschitz domains and hence that $\bu$ has a $C^1$ extension
to $\R^n$.   Fix a point $\bx\oo\in\partial\Omega$.
Then, since $\det\grad\bu(\bx\oo)>0$, $\grad\bu(\bx\oo)$ is invertible.
The inverse function theorem states that (the extension of)
$\bu$ is a diffeomorphism on $B(\bx\oo,r)$ for some $r>0$.
A result of Hofmann, Mitrea, \& Taylor~\cite[Section~4.1]{HMT07}
then shows that $\bu(\Omega)$ is Lipschitz at the point $\bu(\bx\oo)$.
Thus, $\bu(\Omega)$ is a bounded Lipschitz domain.
Finally, we note that, for $1\le p\le\infty$,
\eqref{eqn:chain-rule-1}$_{1,2}$ are each a consequence of
the chain rule
for the composition of a Sobolev function with a diffeomorphism
(see, e.g., \cite[Section~4.26]{Al16}).
\end{proof}


\subsubsection{Body and Surface Forces, the Energy, the Stress, and
the Elasticity Tensor}

We now consider $\bd$, $\bs$, the stored-energy density
$W$ and its first and second derivatives, the Piola-Kirchhoff
stress $\bS$ and the Elasticity Tensor $\A$.  We show how each
transforms from the reference configuration
$\Omc$ to the deformed configuration $\bu(\Omc)$.

\begin{definition}  Given a stored-energy density
$W:\Omc\times\Mnp\to[0,\infty)$
and a diffeomorphism $\bu\in C^1(\Omc;\R^n)$, we define
$\Wu:\bu(\Omc)\times\Mnp\to[0,\infty)$,
\emph{the stored-energy density
with respect to the deformed configuration} $\bu(\Omc)$,
by
\be\label{eqn:chain-0-1}
\Wu(\by,\bG):=W\big(\bx,\bG\bF\big)(\det\bF)^{\mi1},
\ee
\n where $\by=\bu(\bx)$ and $\bF=\bF(\bx):=\grad\bu(\bx)$.
Given a body-force field $\bb\in L^2(\Om;\R^n)$ and, if
$\sS \ne\varnothing$, a surface-traction
field $\bs\in L^2(\sS;\R^n)$ we define
$\bbu:\bu(\Om)\to\R^n$ and $\bsu:\bu(\sS)\to\R^n$,
the body force and surface tractions in the deformed
configuration, by, for $a.e.~\by\in \bu(\Om)$,
\[
\bbu(\by):= \bb(\bx)(\det\bF)^{\mi1},
\qquad
\bsu(\by):=\bs(\bx)|\bF^{\mi\rmT}\bn(\bx)|^{\mi1}(\det\bF)^{\mi1},
\]
\n for $\sH^{n-1}$-$a.e.~\by\in\bu(\sS)$, where $\bn(\bx)$ denotes
the outward unit normal to $\Om$ (which exists at
$\sH^{n-1}$-$a.e.~\bx\in\partial\Om$, since $\partial\Om$
is Lipschitz.)
\end{definition}


The next result is a simple consequence of
the standard chain rule for $C^1$ functions.
\begin{lemma}\label{lem:chain-1-1}
Let $\bu\in C^1(\Omc;\R^n)$ be a diffeomorphism.  Suppose that
$W:\Omc\times\Mnp\to[0,\infty)$ is such that $W$ satisfies (1)--(3)
of Hypothesis~\ref{def:W}.  Then  $W_u$, defined by
\eqref{eqn:chain-0-1}, also satisfies (1)--(3) of
Hypothesis~\ref{def:W}.  Moreover, for $a.e.~\bx\in \Om$,
every $\bG\in\Mnp$, and every $\bH\in\Mn$,
\be\label{eqn:chain-1-1}
\begin{aligned}
{\Su}(\bu(\bx),\bG)\Td\bH
&:= \Big[\frac{\partial}{\partial\bG}\Wu(\bu(\bx),\bG)\Big]\Td\bH =
\bS(\bx,\bG\bF)\Td[\bH\bF](\det\bF)^{\mi1},\\[2pt]
\bH\Td{\Au}(\bu(\bx),\bG)[\bH] &:=
\frac{\partial}{\partial\bG}
\Big({\Su}\big(\bu(\bx),\bG\big)\Td\bH\Big)[\bH]
= [\bH\bF]\Td\A(\bx,\bG\bF)[\bH\bF](\det\bF)^{\mi1},
\end{aligned}
\ee
\n where $\bF=\bF(\bx):=\grad\bu(\bx)$.
\end{lemma}


If we combine  Proposition~\ref{prop:compositions},
Lemma~\ref{lem:chain-1-1},  and the change of variables
formula for injective Lipschitz mappings we conclude the following.
\begin{proposition}\label{prop:cov-1-1}  Let
$\bu\in C^1(\Omc;\R^n)$ be a diffeomorphism.
Suppose that $W$ satisfies (1)--(3) of Hypothesis~\ref{def:W}.
Assume further that $\bvh\in W^{1,\infty}(\bu(\Om);\R^n)$
and $\bwh\in W^{1,2}(\bu(\Om);\R^n)$.
Define
$\bv:=\bvh\circ\bu:\Om\to\R^n$ and
$\bw:=\bwh\circ\bu:\Om\to\R^n$.  Then
$\bv\in W^{1,\infty}(\Om;\R^n)$,  $\bw\in W^{1,2}(\Om;\R^n)$,
and
\be\label{eqn:cov-1-1}
\begin{aligned}
\int_{\Om} W\big(\bx,\grad\bv(\bx)\big) \,\dd\bx
&=
\int_{\bu(\Om)} \Wu\big(\by,\grad\bvh(\by)\big) \,\dd\by,\\
\int_{\Om} \bS\big(\bx,\grad\bv(\bx)\big)\Td\grad\bw(\bx) \,\dd\bx
&=
\int_{\bu(\Om)}
{\Su}\big(\by,\grad\bvh(\by)\big)\Td \grad \bwh(\by) \,\dd\by,\\
\int_{\Om}
\grad\bw(\bx)\Td
\A
\big(\bx,\grad\bv(\bx)\big)[\grad\bw(\bx)] \,\dd\bx
&=
\int_{\bu(\Om)}
\grad \bwh(\by)\Td{\Au}\big(\by,\grad\bvh(\by)\big)
[\grad \bwh(\by)]\,\dd\by,   \\
\int_{\Om}\bb(\bx)\cdot\bw(\bx)\,\dd\bx
&=
\int_{\bu(\Om)}\bbu(\by)\cdot\bwh(\by)\,\dd\by,\\
\int_{\sS}\bs(\bx)\cdot\bw(\bx)\,\dd\sH^{n-1}_\bx
&=
\int_{\bu(\sS)}\bsu(\by)\cdot\bwh(\by)\,\dd\sH^{n-1}_\by.
\end{aligned}
\ee
\end{proposition}


\begin{remark} Equations \eqref{eqn:cov-1-1} remain valid if
$\bv\in W^{1,\infty}(\Om;\R^n)$
and $\bw\in W^{1,2}(\Om;\R^n)$ are prescribed and
$\bvh:=\bv\circ\bu^{\mi1}\in W^{1,\infty}(\bu(\Om);\R^n)$ and
$\bwh:=\bw\circ\bu^{\mi1}\in W^{1,2}(\bu(\Om);\R^n)$ are defined.
\end{remark}

\begin{proof}[Proof of Proposition~\ref{prop:cov-1-1}]
We shall prove \eqref{eqn:cov-1-1}$_2$.  The proofs of
the other equations are similar.\footnote{Equation
\eqref{eqn:cov-1-1}$_5$ is based upon the identities
$\bs=\bS\bn$, $\bsu=\Su\bm$,
$\bm=(\bF^{\mi\rmT}\bn)/|\bF^{\mi\rmT}\bn|$, and
(cf.~\eqref{eqn:chain-1-1}$_1$)  $\bS\bF^\rmT=(\det\bF)\bS_u$,
where $\bm$ denotes the outward unit normal to $\bu(\Om)$.
See, e.g., \cite[Section~1.7]{Ci1988}.}
Let $W$ satisfy (1)--(3) of
Hypothesis~\ref{def:W}
and suppose that $\bu$,  $\bv$, $\bvh$, $\bw$,
and $\bwh$ are as given in the statement of the proposition.
Then, by Proposition~\ref{prop:compositions},
$\bv\in W^{1,\infty}(\Om;\R^n)$, $\bw\in W^{1,2}(\Om;\R^n)$,
and
\be\label{eqn:chain-4-1}
\grad_\bx\bv(\bx)=\grad_\by\bvh\big(\bu(\bx)\big)\grad\bu(\bx),
\qquad
\grad_\bx\bw(\bx)=\grad_\by\bwh\big(\bu(\bx)\big)\grad\bu(\bx),
\ee
\n for $a.e.~\bx\in \Om$.  Therefore, in view of
\eqref{eqn:chain-4-1} and \eqref{eqn:chain-1-1}$_1$ with
$\bG=\grad_\by\bvh(\bu(\bx))$, $\bH=\grad_\by\bwh(\bu(\bx))$,
and $\bF=\grad\bu(\bx)$,
\be\label{eqn:S=Shat-1}
\bS\big(\bx,\grad\bv(\bx)\big)\Td\grad\bw(\bx)=
{\Su}\big(\by,\grad\bvh(\by)\big)
\Td\grad\bwh(\by)
\big[\det\grad\bu(\bx)\big], \quad \by:=\bu(\bx).
\ee
\n Finally, we integrate \eqref{eqn:S=Shat-1} over $\Om$ and
then apply the change of variables formula for  injective Lipschitz
mappings (see, e.g., \cite[Theorem~3.2.5]{Fe69})
to deduce the desired result, \eqref{eqn:cov-1-1}$_2$.
\end{proof}


We now fix a diffeomorphism $\bu\in \AD$  (see
\eqref{eqn:var-and-AD}$_1$) and consider $\bu(\Om)$ as a new reference
configuration.  We first define the admissible deformations and the
corresponding variations that originate at this reference configuration.

\begin{definition}
Fix a \emph{diffeomorphism} $\bu\in C^1(\Omc;\R^n)$
that satisfies $\bu\in\AD$ and define
\[
\begin{gathered}
\ADu:=\{\bvh\in W^{1,\infty}(\bu(\Om);\R^n):
\det\grad\bvh>0~a.e.,
\ \bvh=\bi\ \text{ on } \bu(\sD)\},\\[2pt]
\Varu:=\{\bwh\in W^{1,2}(\bu(\Om);\R^n):
\bwh=\mathbf{0} \text{ on $\bu(\sD)$}\}.
\end{gathered}
\]
\end{definition}


Recall that the total energy $\E$ of $\bv\in\AD$ is defined by
\be\label{eqn:TEr}
\E(\bv):=\int_\Om \big[W\big(\bx,\grad\bv(\bx)\big)
-\bb(\bx)\cdot\bv(\bx)\big]\dd\bx
-\int_{\sS} \bs(\bx)\cdot\bv(\bx)\,\dd\sH^{n-1}_\bx
\ee
\n and $\bv\e\in\AD$ is a weak
solution of the equilibrium equations corresponding
to \eqref{eqn:TEr} if
\be\label{eqn:EEr}
0=\int_\Om
\big[\bS\big(\bx,\grad\bv\e(\bx)\big):\grad\bw(\bx)
-\bb(\bx)\cdot\bw(\bx)\big]\dd\bx
-
\int_\sS \bs(\bx)\cdot\bw(\bx)\,\dd \sH^{n-1}_\bx
\ee
\n for all variations $\bw\in\Var$.


\begin{lemma}\label{lem:ref=deformed-alt}
Let $\bu\in\AD$ be a diffeomorphism and suppose that $\bv\in\AD$.
Then $\bv$ is a weak solution of the equilibrium equations
\eqref{eqn:EEr}  if and only if  $\bvh:= \bv\circ\bu^{\mi1}$
is a weak solution of the equilibrium equations corresponding to
the energy
\be\label{eqn:TEd-alt}
\Eu(\bzh):=\int_{\bu(\Om)} \big[\Wu\big(\by,\grad\bzh(\by)\big)
-\bbu(\by)\cdot\bzh(\by)\big]\dd\by
-\int_{\bu(\sS)} \bsu(\by)\cdot\bzh(\by) \,\dd\sH^{n-1}_\by.
\ee
\n Moreover, necessary and sufficient conditions for the
uniform positivity of the second variation of $\E$ at
$\bv$ is that the second variation of $\Eu$
be uniformly positive at $\bvh$.
\end{lemma}


\begin{proof}  The first assertion follows from
\eqref{eqn:cov-1-1}$_{2,4,5}$, \eqref{eqn:EEr}, and an argument similar
to the following one.
To prove sufficiency, suppose that the second variation of $\Eu$
is uniformly
positive at $\bvh$ with constant $k$.  Fix $\bw\in \Var$ and define
$\bwh:=\bw\circ\bu^{\mi1}$.  Then, by
Proposition~\ref{prop:compositions},
$\bwh\in W^{1,2}(\bu(\Om);\R^n)$ with
\be\label{eqn:chain-again-alt}
\grad\bwh(\by)
=
\grad_\bx \bw\big(\bu^{\mi1}(\by)\big)
\grad \bu^{\mi1}(\by)\
\text{ for $a.e.~\by\in\bu(\Om)$}.
\ee
\n Moreover, since $\bw=\mathbf{0}$ on $\sD$ it follows that
$\bwh=\mathbf{0}$ on $\bd(\sD)$ and hence that $\bwh\in\Varu$.

Next, the assumed uniform positivity together with
\eqref{eqn:cov-1-1}$_3$
shows that the second variation of $\E$ at $\bv$ in
the direction $\bw$ is bounded below by
\[
k\int_{\bu(\Om)} |\grad\bwh(\by)|^2\,\dd\by =
k\int_{\Om} |\grad\bw(\bx)|^2\det\grad\bu(\bx)\,\dd\bx,
\]
\n where the last equality follows from \eqref{eqn:chain-again-alt} and
the change of variables formula.  The desired result now follows
since $\det\grad\bu$ is bounded away from zero on the compact set
$\Omc$.  The necessity argument is similar.
\end{proof}


\subsection{Proof of Theorem~\ref{thm:NT-3}} \label{sec:Proof}

Our proof of Theorem~\ref{thm:NT-3} will require us to show that
\eqref{eqn:dist-small-ref} implies that the gradient of some mapping is
sufficiently close to the set of rotations.
We first define this mapping and show that the distance of its gradient
from the rotations is bounded above by a constant times
the strain difference given in \eqref{eqn:dist-small-ref}.

\begin{lemma}\label{lem:strain=dist}  Let  $\bu\e,\bv\in\AD$ with
$\bu\e$ a diffeomorphism. Define $\bFFe:=\grad\bue$,
$\bGG:=\grad\bv$,
\be\label{eqn:sup+inf+d}
\Upsilon_{\!\mathrm e}:= \sup_{\bx\in\Omc} |\bFFe(\bx)|,
\qquad
\upsilon\e:=\inf_{\bx\in\Omc} \big|[\bFFe(\bx)]^{\mi1}\big|^{\mi1},
\qquad
d(\bx):=\dist\!\big(\bGG\bFFei,\SOn\big).
\ee
\n Then
\be\label{eqn:strains-vs-dist-to-rotations}
\upsilon\e^{2} d^2\le
\sqrt{n}\,|\bGGT\bGG-\bFFeT\bFFe|
\le \Upsilon_{\!\mathrm e}^2 d\sqrt{n}\,\big(d+2\sqrt{n}\,\big).
\ee
\end{lemma}


\begin{proof}
We first note that
$\grad\bu\e\in C^0(\Omc;\R^n)$ with
$\det\grad\bu\e>0$ on the compact set $\Omc$ and hence $\Upsilon\e$ and
$\upsilon\e$  are strictly positive and finite. Define
\[
\bF=\bF(\bx):= \bGG\bFFei, \qquad
\bE=\bE(\bx):=\tfrac12\big(\bF^\rmT\bF-\bI\big).
\]
\n Then Lemma~\ref{lem:strain=dist-to-rotations}
shows that $\bE$ and $d$, given by \eqref{eqn:sup+inf+d}$_3$, satisfy
\be\label{eqn:L7.6-2}
d^2 \le 2\sqrt{n}\,|\bE|\le  d\sqrt{n}\,\big(d +2\sqrt{n}\,\big).
\ee


  Next, consider
\begin{gather}\label{eqn:equals}
\bGGT\bGG-\bFFeT\bFFe
=\bFFeT
\big[(\bGG\bFFei)^{\rmT}
\bGG\bFFei-\bI\big]
\bFFe
=2\bFFeT\bE\bFFe,\\[2pt]
|\bE|
=
\big|\bFFeiT
\big(\bFFeT\bE\bFFe\big)
\bFFei\big|
\le
\big|\bFFei\big|^2
\big|\bFFeT\bE\bFFe\big|, \qquad
\big|\bFFeT\bE\bFFe\big|
\le |\bE||\bFFe|^2.
\label{eqn:details}
\end{gather}
\n Thus, if we now combine \eqref{eqn:L7.6-2} and \eqref{eqn:details}
we find that
\be\label{eqn:last-strain-proof}
\tfrac12\tfrac{d^2}{\sqrt{n}\,}|\bFFei|^{\mi2}
\le |\bFFei|^{\mi2}|\bE|
\le  \big|\bFFeT\bE\bFFe\big|
\le |\bE||\bFFe|^2
\le \tfrac12d\big(d +2\sqrt{n}\,\big) |\bFFe|^2.
\ee
\n Finally, \eqref{eqn:sup+inf+d}$_{1,2}$,  \eqref{eqn:equals},
and \eqref{eqn:last-strain-proof}
yield the desired result, \eqref{eqn:strains-vs-dist-to-rotations}.
\end{proof}

\begin{remark}\label{rem:near-rotations} (1).~The mapping whose gradient
is close to the set of rotations is $\bv\circ\bu\e^{\mi1}$.  (2).~If we
make use of \eqref{eqn:d<2E}, in place of the first inequality in
Lemma~\ref{lem:strain=dist-to-rotations}, we find that
\be\label{eqn:d<2C-Ce}
\upsilon\e^{2} d\le
|\bGGT\bGG-\bFFeT\bFFe|.
\ee
\n  Once again, although \eqref{eqn:d<2C-Ce} does not scale properly for
large $d$, its use will simplify the computation, which involves
small-strain differences, in the next proof.
\end{remark}


\begin{proof}[Proof of Theorem~\ref{thm:NT-3}]
Let $\bu\e\in\AD$ satisfy hypotheses (A)--(C) of the theorem.   Define
$\Ome:=\bu\e(\Om)$. Then, by Proposition~\ref{prop:compositions},
$\Ome$ is a bounded Lipschitz domain.  Suppose that $\varepsilon>0$ is a
small parameter to be determined and let
$\bv\in \AD$  satisfy \eqref{eqn:dist-small-ref}.
Define, $\buh\e,\bvh\in\AD_{u_e}$ by
\be\label{eqn:def-of-defs}
\buh\e:=\bu\e\circ\bu\e^{\mi1}=\bi, \qquad \bvh:=\bv\circ\bu\e^{\mi1}.
\ee


Let $\delta\in(0,1)$ be given as in Theorem~\ref{thm:SVP=LMBMO}.
We shall determine $\varepsilon$ such that $\buh\e$ and $\bvh$ satisfy
the hypotheses of Theorem~\ref{thm:SVP=LMBMO}  (with $\bu\e,\bv,\Om,$
and $\E$ replaced by  $\buh\e,\bvh,\Ome,$ and $\E_{u_e}$).
In view of Lemma~\ref{lem:ref=deformed-alt} and assumptions (A)--(C),
$\buh\e=\bi$ is a weak equilibrium solution for
$\E_{u_e}$, given by \eqref{eqn:TEd-alt} with $u=u_e$, at which the second
variation of $\E_{u_e}$ is uniformly positive.  Thus, $\buh\e$ satisfies
\eqref{eqn:small+SV}$_1$.
Trivially,
$\dist(\bI,\SOn)=0$ and
the rotation associated with $\bi$ in
Proposition~\ref{prop:GR}
is $\bI$.  Next, if we combine \eqref{eqn:dist-small-ref}
and \eqref{eqn:d<2C-Ce} we find, with the aid of \eqref{eqn:der-inv},
\eqref{eqn:def-of-defs}$_2$, Remark~\ref{rem:a.e.}, and the chain
rule, that
\be\label{eqn:dist-Dv-small}
\dist\!\big(\grad\bvh(\by),\SOn\big)
<\upsilon\e^{\mi2}\varepsilon\
\text{ for $a.e.~\by\in\Ome$.}
\ee
\n Next, fix $p>n$. Then \eqref{eqn:dist-Dv-small},
Proposition~\ref{prop:BC+SR=close-in-L1}, and
Proposition~\ref{prop:GR} yield
\[  
||\grad\bvh-\bI||_{1,\Ome}
<
\upsilon\e^{\mi2}A^*|\Ome|^{1/p}\varepsilon.
\qquad
\tsn{\grad\bvh-\bI}_{\BMO(\Ome)} =
\tsn{\grad\bvh}_{\BMO(\Ome)} < M\upsilon\e^{\mi2}\varepsilon
\]  
\n for some constants $A^*>0$ and $M>0$.
%


Now, let  $\varepsilon>0$ be sufficiently small so that
\[
\max\{M\varepsilon,\
\varepsilon,\
A^*|\Ome|^{1/p}\varepsilon \}<\upsilon\e^{2}\delta.
\]
\n In addition, define
\[
\sB:=\{\bG\in\Mnp: \dist(\bG,\SOn) < \delta < 1 \}.
\]
\n Then the hypotheses of Theorem~\ref{thm:SVP=LMBMO}
have been satisfied; consequently that result yields
\be\label{eqn:estimate-E}
\E_{u_e}(\bvh)\ge \E_{u_e}(\bi)
+k\int_\Ome|\grad\bvh-\bI|^2\dd\by.
\ee
\n Moreover, $\bvh$ \emph{cannot be a weak solution of the
equilibrium equations corresponding to $\E_{u_e}$}.

Next, Proposition~\ref{prop:cov-1-1} together with \eqref{eqn:TEr}
and \eqref{eqn:TEd-alt}  shows that
$\E_{u_e}(\bvh)=\E(\bv)$ and $\E_{u_e}(\bi)=\E(\bu_e)$;  thus,
by \eqref{eqn:estimate-E},
\[
\E(\bv)\ge \E(\bu\e)
+k\int_\Ome|\grad\bvh-\bI|^2\dd\by.
\]
\n Consequently, $\E(\bv)> \E(\bu\e)$
unless $\grad\bvh\equiv\bI$.
However, $\grad\bvh=\bI$ on the connected open set $\Ome$
together with $\bvh=\bi$ on $\sD$ yields $\bvh\equiv\bi$.  Equivalently,
(cf.~\eqref{eqn:def-of-defs}$_2$) $\bv\circ\bu\e^{\mi1}=\bi$ and so
$\bv=\bu\e$. Therefore, $\bv\not\equiv\bu\e$ will
have strictly greater energy than $\bu\e$.


Finally, if $\bv$ were to satisfy \eqref{eqn:EEr},
then Lemma~\ref{lem:ref=deformed-alt}
would imply that $\bvh=\bv\circ\bu\e^{\mi1}$ is a weak solution of
the equilibrium equations corresponding to $\E_e$.  However, this is not
possible (see the sentence in italics following \eqref{eqn:estimate-E}).
\end{proof}

\n \textbf{Acknowledgement.} The authors thank Mario Milman for
interesting discussions regarding $\BMO$ and interpolation theory.
The authors would also like to thank one of the referees of
\cite{SS18} for their suggestion that the results in
Kristensen \& Taheri~\cite{KT03} might
lead to an extension of John's~\cite{Jo72}
uniqueness theorem to the mixed problem.

\appendix


\section{Versions of Taylor's Theorem for
Non-convex Sets}\label{sec:Taylor}

Recall that $\Omega\subset\R^n$, $n\ge2$, is a Lipschitz domain and
$\cO\subset\MN$ is a nonempty, open set.
If $\sB\subset\MN$ is a nonempty, bounded,
open set that satisfies $\overline{\sB}\subset\cO$,  then, for
$\varepsilon>0$ and sufficiently small, the set
\be\label{eqn:neigh-of-sB-III}
\sB_\varepsilon := \{\bK \in \MN: |\bK-\bF|<\varepsilon\
\text{ for some $\bF\in\sB$}\}
\ee
\n is a nonempty, bounded, open set that satisfies
$\overline{\sB}_\varepsilon \subset \cO$.


\begin{lemma}\label{lem:taylor-III} Let $\Omega$, $\cO$, and
$W:\overline{\Omega}\times\cO\to\R$ be as given in
(1)--(3) of Hypothesis~\ref{def:W}.
Suppose that $\sB\subset\MN$ is a nonempty,
bounded, open set that satisfies $\overline{\sB}\subset\cO$.
Then there exists a constant $c=c(\sB)>0$ such that, for every
$\bF,\bG\in\overline{\sB}$ and almost every $\bx\in\Omega$,
\be\label{eqn:taylor-III-1}
W(\bx,\bG)\ge W(\bx,\bF) + \DD W(\bx,\bF)[\bH]
+\tfrac12\DD^2 W(\bx,\bF)[\bH,\bH] -c|\bH|^3,
\ee
\n where $\bH:=\bG-\bF$.
\end{lemma}

\begin{remark} If $\sB$ is convex, then Lemma~\ref{lem:taylor-III}
follows from Taylor's theorem.
\end{remark}


\begin{proof}[Proof of Lemma~\ref{lem:taylor-III}] Given
$\sB\subset\overline{\sB}\subset\cO$, let $\sB_\varepsilon$
(defined by \eqref{eqn:neigh-of-sB-III}) satisfy
$\overline{\sB}_\varepsilon \subset \cO$.  Define
\be\label{eqn:taylor-III-2}
c:= \sup_{\substack{\bF\in \overline{\sB},\,\bx\in\overline{\Omega}\\
\bG\in \overline{\sB}_\varepsilon}}
\frac{W(\bx,\bF)-W(\bx,\bG)
+ \DD W(\bx,\bF)[\bH]
+\tfrac12\DD^2 W(\bx,\bF)[\bH,\bH]}{|\bH|^3},
\ee
\n where $\bH:=\bG-\bF$. We need only show that the supremum is
finite in
order to conclude that \eqref{eqn:taylor-III-1} is satisfied for all
$\bF,\bG\in\overline{\sB}$ and  $a.e.~\bx\in\Omega$.  Suppose
that the right-hand side of \eqref{eqn:taylor-III-2} is not bounded.
In view of (3) of Hypothesis~\ref{def:W},
the numerator in \eqref{eqn:taylor-III-2} is bounded on the
compact set
$\overline{\Omega}\times\overline{\sB}_\varepsilon\times\overline{\sB}$;
thus, there must exist sequences $\bx_k\in\overline{\Omega}$,
$\bF_k\in\overline{\sB}$, and $\bG_k\in\overline{\sB}_\varepsilon$
such that $\bH_k:=\bG_k-\bF_k\to\mathbf{0}$.  It follows that there
exists
$\bP\in\overline{\sB}$ such that, for a subsequence (not relabeled)
$\bF_k,\bG_k\to\bP$.


We note that $\bP\in\sB_\varepsilon$, an open set; thus exists
a $\delta>0$ such that the open ball of radius $2\delta$
centered at $\bP$,
$B(\bP,2\delta)\subset\sB_\varepsilon$. Then, for
$k$ sufficiently large, $\bF_k,\bG_k\in B(\bP,\delta)$. In
addition,  since
$\bF\mapsto W(\cdot,\bF)$ is $C^3$,
$\overline{\Omega}\times\overline{B}(\bP,\delta)$ is compact,
and the unit ball in $\MN$ is compact,
it follows from (3) of Hypothesis~\ref{def:W} that
\[ 
\begin{gathered}
c^*:= \sup_{\substack{\bx\in \overline{\Omega} \\
\bN\in\overline{B}(\bP,\delta)}} |\DD^3 W(\bx,\bN)| <\infty,\
\text{ where }\\
|\DD^3 W(\bx,\bN)|:=\sup_{\substack{|\bK|\le1
\\ |\bL|\le1,\, |\bR|\le1}}\big|\DD^3 W(\bx,\bN)[\bK,\bL,\bR]\big|.
\end{gathered}
\] 

Next, choose $k\oo$ such that $\bF_k,\bG_k\in B(\bP,\delta)$,
for all $k\ge k\oo$, and apply Taylor's
theorem (see, e.g., \cite[Section~4.6]{Ze86}) to the function
$\bF\mapsto W(\bx_k,\bF)$ at $\bF_k$ and $\bG_k$
to conclude that, for all $k\ge k\oo$,
\be\label{eqn:taylor-III-3}
\begin{aligned}
W(\bx_k,\bG_k)= W(\bx_k,\bF_k) &+ \DD W(\bx_k,\bF_k)[\bH_k]\\[4pt]
&+\tfrac12\DD^2 W(\bx_k,\bF_k)[\bH_k,\bH_k]
+\tfrac16R(\bx_k,\bF_k,\bH_k),
\end{aligned}
\ee
\n where $\bH_k:=\bG_k-\bF_k$ and
\be\label{eqn:remainder-III}
|R(\bx_k,\bF_k,\bH_k)|
\le
|\bH_k|^3\sup_{t\in[0,1]}|\DD^3 W(\bx_k,\bF_k+t\bH_k)|\le c^*|\bH_k|^3.
\ee
\n  Then, in view of
\eqref{eqn:taylor-III-3} and \eqref{eqn:remainder-III},
\[
W(\bx_k,\bF_k)-W(\bx_k,\bG_k)
+ \DD W(\bx_k,\bF_k)[\bH_k]
+\tfrac12\DD^2 W(\bx_k,\bF_k)[\bH_k,\bH_k]
\le
\tfrac16 c^*|\bH_k|^3.
\]
\n This contradicts our assumption that the right-hand side of
\eqref{eqn:taylor-III-2} becomes arbitrarily large when $\bF=\bF_k$,
$\bG=\bG_k$, $\bx=\bx_k$, and $k\to\infty$.
\end{proof}


\begin{lemma}\label{lem:taylor-III-AA} Let $\Omega$, $\cO$,
$W:\overline{\Omega}\times\cO\to\R$, and $\sB\subset\MN$  be as given in
the statement of Lemma~\ref{lem:taylor-III}.
Then there exists a constant $\hat{c}=\hat{c}(\sB)>0$ such that,
for every
$\bF,\bG\in\overline{\sB}$, every $\bL\in\MN$,
and almost every $\bx\in\Omega$,
\[ 
\DD^2 W(\bx,\bG)[\bL,\bL]
\ge
\DD^2 W(\bx,\bF)[\bL,\bL] -\hat{c}|\bG-\bF||\bL|^2.
\] 
\end{lemma}

The proof of the above result is similar to the proof of
Lemma~\ref{lem:taylor-III} with the constant $\hat{c}$ now
given by
\[
\hat{c}:=
\sup_{\substack{\bF\in \overline{\sB},\,\bx\in\overline{\Omega}\\
\bG\in \overline{\sB}_\varepsilon,\,|\bK|=1}}
\frac{\DD^2 W(\bx,\bF)[\bK,\bK]-\DD^2 W(\bx,\bG)[\bK,\bK]}{|\bG-\bF|}.
\]


\section{A Generalized Korn Inequality}\label{sec:Korn}

Our first result in Section~\ref{sec:ES-EM} required a more general version
of Korn's inequality than is usually needed in Nonlinear Elasticity.
The precise version we used can be found in a paper of
Pompe~\cite[Corollary~4.1]{Po03}.

\begin{proposition}\label{prop:Korn-1}\emph{(Korn's Inequality with
Variable Coefficients)}
Let $\bF\in C(\overline{\Omega};\Mn)$  satisfy
$\det\bF(\bx)\ge\mu>0$ for all $\bx\in\overline{\Omega}$.  Then there
exists a constant $K>0$ such that
\be\label{eqn:Korn-1}
\int_\Omega
\Big|\big[\bF(\bx)\big]^\rmT\grad\bw(\bx)
+
\big[\grad\bw(\bx)\big]^\rmT\bF(\bx)\Big|^2\,\dd\bx
\ge
K \int_\Omega \big|\grad\bw(\bx)\big|^2\,\dd\bx,
\ee
\n for every $\bw\in W^{1,2}(\Omega;\R^n)$ that satisfies
$\bw=\mathbf{0}$ on $\sD$.
\end{proposition}


\begin{remark}  The standard version of Korn's inequality
occurs when $\bF(\bx)\equiv\bI$ in \eqref{eqn:Korn-1}.
Proposition~\ref{prop:Korn-1} is \emph{not} generally valid
if one assumes only that $\bF\in L^\infty(\Omega;\Mn)$.
Counterexamples can be found in Neff \& Pompe~\cite{NP14} and
the references therein.  Proposition~\ref{prop:Korn-1} can also
be obtained\footnote{If the boundary is $C^1$, then
\eqref{eqn:Korn-1} is a consequence of results of
de~Figueiredo~\cite{deF63}.} from results of
Hlav\'a\v cek \& Ne\v cas~\cite{HN70}
that address the problem
of coercivity for formally positive quadratic forms of
vector-valued functions (e.g., the left-hand side of
\eqref{eqn:Korn-1}).  However, \cite{HN70} does not establish
precisely \eqref{eqn:Korn-1}.
\end{remark}



\begin{thebibliography}{99}

\bibitem{AF03} Adams, R. A.,  Fournier, J. J. F.:
Sobolev Spaces. $2^{nd}$ edition.
Elsevier/Academic Press, Amsterdam, (2003)

\bibitem{Al16} Alt, H. W.: Linear Functional Analysis.
An Application-oriented Introduction.
Translated from the German edition by Robert N\"urnberg.
Springer, London,  (2016)

\bibitem{An79} Antman, S. S.:
The eversion of thick spherical shells.
Arch. Ration. Mech. Anal.
\textbf{70}, 113--123 (1979)

\bibitem{Ba77}  Ball, J. M.:
Convexity conditions and existence theorems in nonlinear elasticity.
Arch. Ration. Mech. Anal.
\textbf{63}, 337--403 (1977)

\bibitem{Ba02}  Ball, J. M.:
Some open problems in elasticity.
In: Newton, P., Holmes, P., Weinstein, A. (eds.)
Geometry, Mechanics, and Dynamics,
pp. 3--59, Springer, New York (2002)

\bibitem{BS88} Bennett, C., Sharpley, R.:
Interpolation of Operators.
Academic Press, Inc., Boston, MA  (1988)

\bibitem{BL00} Benyamini, Y.,  Lindenstrauss, J.:
Geometric nonlinear functional analysis. Vol. 1.
American Mathematical Society, Providence, RI,  (2000)

\bibitem{Be11} Bevan, J. J.:
Extending the Knops-Stuart-Taheri technique to $C^1$
weak local minimizers in nonlinear elasticity.
Proc. Amer. Math. Soc. \textbf{139}, 1667--1679 (2011)

\bibitem{BN95} Brezis, H., Nirenberg, L.;
Degree theory and BMO. I. Compact manifolds without boundaries.
Selecta Math. (N.S.)  \textbf{1},  197--263 (1995)



\bibitem{Ca17} Campos Cordero, J.:
Boundary regularity and sufficient conditions for strong
local minimizers.
J. Funct. Anal. \textbf{272}, 4513--4587 (2017)

\bibitem{CMM09} Carbonaro, A.,  Mauceri, G.,  Meda, S.:
$H^1$ and $\BMO$ for certain locally doubling metric measure spaces.
Ann. Sc. Norm. Super. Pisa Cl. Sci. (5) \textbf{8}, 543--582 (2009)

\bibitem{CMM10} Carbonaro, A.,  Mauceri, G.,  Meda, S.:
$H^1$ and $\BMO$ for certain locally doubling metric measure spaces
of finite measure.  Colloq. Math. \textbf{118}, 13--41 (2010)

\bibitem{CPV03} Carillo, S., Podio-Guidugli, P.,
Vergara Caffarelli, G.:
Second-order surface potentials in finite elasticity.
In: Podio-Guidugli, P., Brocato, M. (eds.)
Rational Continua, Classical and New,
pp. 19--38, Springer Italia, Milan, (2003)

\bibitem{Ci1988} Ciarlet, P. G.:
Mathematical Elasticity, vol. I.,
Elsevier, Amsterdam (1988)

\bibitem{CM15} Ciarlet, P. G., Mardare, C.:
Nonlinear Korn inequalities.
J. Math. Pures Appl. \textbf{104}, 1119--1134 (2015)

\bibitem{CS06} Conti, S.,  Schweizer, B.:
Rigidity and gamma convergence for solid-solid phase
transitions with SO($2$) invariance.
Commun. Pure Appl. Math. \textbf{59}, 830--868  (2006)

\bibitem{CDM14} Conti, S.,  Dolzmann, G.,  M\"uller, S.:
Korn's second inequality and geometric rigidity with mixed growth
conditions.  Calc. Var. Partial Differential Equations
\textbf{50}, 437--454  (2014)

\bibitem{deF63} de Figueiredo, D. G.:
The coerciveness problem for forms over vector valued functions.
 Commun. Pure Appl. Math.  \textbf{16},  63--94   (1963)


\bibitem{DRS10} Diening, L.,
R$\overset{_\circ}{\mathrm{u}}$\v zi\v cka, M.,  Schumacher, K.:
A decomposition technique for John domains.
Ann. Acad. Sci. Fenn. Math.  \textbf{35}, 87--114  (2010)

\bibitem{EG92} Evans L. C., Gariepy, R. F.:
Measure Theory and Fine Properties of Functions.
CRC Press, Boca Raton (1992)

\bibitem{Fe69} Federer, H.:
Geometric Measure Theory.
Springer, New York  (1969)

\bibitem{FS72} Fefferman, C., Stein, E. M.:
$H^p$ spaces of several variables.
Acta Math. \textbf{129},  137--193  (1972)

\bibitem{DFG12} Fefferman, C., Damelin, S. B., Glover, W.:
A $\BMO$ theorem for $\varepsilon$-distorted diffeomorphisms on $\R^D$
and an application to comparing manifolds of speech and sound.
Involve \textbf{5},  159--172 (2012)

\bibitem{Fi92} Firoozye, N. B.:
Positive second variation and local minimizers in $\BMO$-Sobolev spaces.
Preprint no. 252, 1992, SFB 256, University of Bonn

\bibitem{FG95} Fonseca, I.,  Gangbo, W.:
Degree Theory in Analysis and Applications.
The Clarendon Press, Oxford University Press, New York (1995)

\bibitem{FJM02} Friesecke, G., James, R. D.,  M{\"u}ller, S.:
A theorem on geometric rigidity and the derivation of nonlinear
plate theory from three-dimensional elasticity.
Commun. Pure Appl. Math. \textbf{55}, 1461--1506  (2002)

\bibitem{GNRT17} Gao, D., Neff, P., Roventa, I.,  Thiel, C.:
On the convexity of nonlinear elastic energies in the
right Cauchy-Green tensor.
J. Elast. \textbf{127}, 303--308 (2017)

\bibitem{Gr08} Grafakos, L.:
Classical Fourier analysis. $3^{nd}$ edition. Springer, New York (2014)

\bibitem{Gr09} Grafakos, L.:
Modern Fourier analysis. $3^{nd}$ edition. Springer, New York (2014)
		
\bibitem{Gr85} Grisvard, P.:
Elliptic Problems in Nonsmooth Domains.
Pitman, Boston, MA  (1985)


\bibitem{Gu81} Gurtin, M. E.:
An Introduction to Continuum Mechanics.
Academic Press, New York (1981)

\bibitem{GS79} Gurtin, M. E., Spector, S. J.:
On stability and uniqueness in finite elasticity.
Arch. Ration. Mech. Anal. \textbf{70}, 153--165 (1979)

\bibitem{HN70}  Hlav\'a\v cek, I., Ne\v cas, J.:
On inequalities of Korn's type. {I}. {B}oundary-value
problems for elliptic system of partial differential equations.
Arch. Rational Mech. Anal.  \textbf{36},  305--311  (1970)

\bibitem{HMT07} Hofmann, S.,  Mitrea, M., Taylor, M.:
Geometric and transformational properties of Lipschitz domains,
Semmes-Kenig-Toro domains, and other classes of finite
perimeter domains. J. Geom. Anal.  \textbf{17}, 593--647 (2007)


\bibitem{Iw82} Iwaniec, T.:
On $L^p$-integrability in PDEs and quasiregular mappings for large
exponents.
Ann. Acad. Sci. Fenn. Ser. A I Math. \textbf{7}, 301--322  (1982)

\bibitem{IM01} Iwaniec, T., Martin, G.:
Geometric Function Theory and Non-linear Analysis.
The Clarendon Press, Oxford University Press, New York  (2001)


\bibitem{Jo61} John, F.:
Rotation and strain.  Commun. Pure Appl. Math.
\textbf{14}, 391--413 (1961)

\bibitem{Jo72}  John, F.:
Uniqueness of non-linear elastic equilibrium for prescribed boundary
displacements and sufficiently small strains.
Commun. Pure Appl. Math.  \textbf{25}, 617--634  (1972)

\bibitem{Jo72-2}  John, F.: Bounds for deformations in terms of average
strains. Inequalities, III (Proc. Third Sympos., Univ. California,
Los Angeles, Calif., 1969),
pp. 129--144. Academic Press, New York,  (1972)

\bibitem{JN61}	John, F.,  Nirenberg, L.:
On functions of bounded mean oscillation.
Commun. Pure Appl. Math.  \textbf{14}, 415--426 (1961)

\bibitem{Jo82} Jones, P. W.;
Extension theorems for BMO.
Indiana Univ. Math. J.  29,  41--66 (1980)

\bibitem{Ko82} Kohn, R. V.:
New integral estimates for deformations in terms of their
nonlinear strains.
Arch. Ration. Mech. Anal. \textbf{78}, 131--172 (1982)

\bibitem{KS84} Knops, R. J., Stuart, C. A.:
Quasiconvexity and uniqueness of equilibrium solutions in nonlinear
elasticity.
Arch. Ration. Mech. Anal. \textbf{86}, 233--249 (1984)


\bibitem{KT03} Kristensen, J.,  Taheri, A.:
Partial regularity of strong local minimizers in the
multi-dimensional calculus of variations.
Arch. Ration. Mech. Anal.  \textbf{170}, 63--89 (2003)

\bibitem{Le09} Leoni, G.: A First Course in Sobolev Spaces.
American Mathematical Society, Providence, RI  (2009)

\bibitem{LL01} Lieb, E. H.,  Loss, M.
Analysis. Second edition.
American Mathematical Society, Providence, RI  (2001)

\bibitem{MS79} Martio, O., Sarvas, J.:
Injectivity theorems in plane and space.
Ann. Acad. Sci. Fenn. Ser. A I Math.  \textbf{4}, 383--401 (1979)

\bibitem{Mo66} Morrey, C. B., Jr.:
Multiple Integrals in the Calculus of Variations.
Springer, New York  (1966)

\bibitem{NP14} Neff, P.,  Pompe, W.:
Counterexamples in the theory of coerciveness for linear elliptic
systems related to generalizations of Korn's second inequality.
Z. Angew. Math. Mech.  \textbf{94}, 784--790 (2014)

\bibitem{Og84} Ogden, R. W.:
Non-Linear Elastic Deformations.
Dover (1984)


\bibitem{PV90} Podio-Guidugli, P., Vergara-Caffarelli, G.:
Surface interaction potentials in elasticity.
Arch. Ration. Mech. Anal.
\textbf{109}, 343--383 (1990)

\bibitem{Po03} Pompe, W.:
Korn's first inequality with variable coefficients
and its generalization.
Comment. Math. Univ. Carolin.  \textbf{44}, 57--70  (2003)

\bibitem{PS97} Post, K. D. E., Sivaloganathan, J.:
On homotopy conditions and the existence of multiple equilibria in
finite elasticity.
Proc. Roy. Soc. Edinburgh Sect. A  \textbf{127},
595--614 (1997)   [Erratum: \textbf{127}, 1111  (1997)]

\bibitem{Re67} Re\v setnjak, J. G.:
Liouville's conformal mapping theorem under minimal
regularity hypotheses.
(Russian)  Sibirsk. Mat. \v Z.  \textbf{8},  835--840  (1967)
[English Translation: Siberian Math. J.  \textbf{8}, 631--634 (1967)]

\bibitem{Se67} Sewell, M. J.:
On configuration-dependent loading.
Arch. Ration. Mech. Anal. \textbf{23}, 327--351 (1967)

\bibitem{Si97}  {\v S}ilhav{\' y}, M.:
The Mechanics and Thermodynamics of Continuous Media.
Springer, Berlin (1997)

\bibitem{SS18} Sivaloganathan, J., Spector, S. J.:
On the uniqueness of energy minimizers in finite
elasticity.  To Appear: Journal of Elasticity (2018)
https://doi.org/10.1007/s10659-018-9671-8


\bibitem{Sp80} Spector, S. J.:
On uniqueness in finite elasticity with general loading.
J. Elast. \textbf{10}, 145--161 (1980)

\bibitem{Sp65}  Spivak, M.
Calculus on Manifolds. W. A. Benjamin Inc.,
New York-Amsterdam, (1965)

\bibitem{Stein-1970}  Stein, E. M.:
Singular Integrals and Differentiability Properties of
Functions.  Princeton University Press, Princeton, N.J. (1970)

\bibitem{Stein-1993} Stein, E. M.:
Harmonic Analysis: Real-variable Methods, Orthogonality, and
Oscillatory Integrals.
Princeton University Press, Princeton, NJ  (1993)


\bibitem{Ta03} Taheri, A.:
Quasiconvexity and uniqueness of stationary points in
the multi-dimensional calculus of variations.
Proc. Amer. Math. Soc.  \textbf{131}, 3101--3107 (2003)

\bibitem{TN65} Truesdell, C.,  Noll, W.:
The non-linear field theories of mechanics.
Handbuch der Physik, Band III/3.
Springer, Berlin (1965)

\bibitem{Va88} Valent, T.:
Boundary Value Problems of Finite Elasticity.
Springer, New York (1988)

\bibitem{VZ08} Verde, A., Zecca, G.:
On the higher integrability for certain nonlinear problems.
Differential Integral Equations  \textbf{21}, 247--263  (2008)

\bibitem{Wh34}  Whitney, H.:
Functions differentiable on the boundaries of regions.
Ann. of Math. (2)  \textbf{35}, 482--485 (1934)


\bibitem{Ze86} Zeidler, E.: Nonlinear Functional
Analysis and its Applications. I.
Fixed-point theorems. Translated from the German by P. R. Wadsack.
Springer, New York (1986)

\bibitem{Zh91} Zhang, K.:
Energy minimizers in nonlinear elastostatics and the implicit
function theorem. Arch. Ration. Mech. Anal.
\textbf{114}, 95--117 (1991)


\end{thebibliography}
\end{document}